\documentclass[11pt]{article}
\usepackage{amssymb}
\usepackage{amsmath}
\usepackage{amsthm}
\usepackage{etoolbox}
\usepackage{graphicx}
\usepackage{soul}
\usepackage[numbers]{natbib}
\usepackage[usenames,dvipsnames,svgnames,table]{xcolor}
\usepackage[pdfborder={0 0 0}]{hyperref}
\usepackage[textwidth=16cm, textheight=24.2cm]{geometry}

\usepackage[english]{babel}
\usepackage[utf8]{inputenc}
\usepackage[T1]{fontenc}
\usepackage{todonotes}
\usepackage{xcolor}
\usepackage{soul}

\usepackage{longtable}

\newtheorem*{thm*}{Theorem}
\newtheorem{thm}{Theorem}[section]
\newtheorem{cor}[thm]{Corollary}
\newtheorem{lem}[thm]{Lemma}

\newtheorem{prop}[thm]{Proposition}
\newtheorem{defn}[thm]{Definition}
\newtheorem{remark}[thm]{Remark}
\newtheorem{example}[thm]{Example}
\newtheorem{mainthm}{Theorem}

\def\dimk{\operatorname{dim}_k}

\def\Spec{\operatorname{Spec}}
\def\Proj{\operatorname{Proj}}

\def\pp#1{\left(#1\right)}

\def\annn#1#2{\operatorname{ann}_{#1}\pp{#2}}

\def\tensor{\otimes}
\def\onto{\twoheadrightarrow}

\def\mc#1{\mathcal{#1}}

\def\gr{\operatorname{gr}}
\def\len{\operatorname{len}}
\def\Hilb{\mathcal{H}\hspace{-0.25ex}\mathit{ilb\/}}

\def\hook{\lrcorner}
\def\rk{\operatorname{rk}}
\def\HilbGor{\mathcal{H}\hspace{-0.25ex}\mathit{ilb\/}^{G}}%
\def\HilbGorGen{\mathcal{H}\hspace{-0.25ex}\mathit{ilb\/}^{G, gen}}%

\def\DPut#1#2{\global\csdef{D#1}{#2}#2}
\def\DDef#1#2{\global\csdef{D#1}{#2}}

\begin{document}
\def\Apolar#1{\operatorname{Apolar}\left( #1 \right)}

\setstcolor{red}

\title{Irreducibility of the Gorenstein loci of Hilbert schemes via ray
families}
\author{Gianfranco Casnati,
Joachim Jelisiejew,
Roberto Notari\thanks{The first and third authors are supported by the framework of PRIN
2010/11 ``Geometria delle variet\`a algebriche'', cofinanced by MIUR, and are members of GNSAGA of INdAM. The
second author was partially supported by the project ``Secant varieties, computational
complexity, and toric degenerations''
realised within the Homing Plus programme of Foundation for Polish Science, co-financed from European Union, Regional
Development Fund. The second author is a doctoral fellow at the Warsaw Center of
Mathematics and Computer Science financed by the Polish program KNOW.
This paper is a part of ``Computational complexity,
    generalised Waring type problems and tensor decompositions'' project
    within ``Canaletto'',  the executive program for scientific and
technological cooperation between Italy and Poland, 2013--2015. This article
is partially supported by Polish National Science Center, project
2014/13/N/ST1/02640.}}
\maketitle

\begin{abstract}
    We analyse the Gorenstein locus of the Hilbert scheme of $d$ points on $\mathbb{P}^n$ i.e.~the open
    subscheme parameterising zero-dimensional Gorenstein subschemes of
    $\mathbb{P}^n$ of degree $d$.
    We give new sufficient criteria for smoothability and smoothness of points of
    the Gorenstein locus.  In particular we
    prove that this locus is irreducible when $d\leq 13$ and find
    its components when $d = 14$.

    The proof is relatively self-contained and it does not rely on a computer
    algebra system.  As a by--product, we give equations of the
    fourth
    secant variety to the $d$-th Veronese reembedding of $\mathbb{P}^n$ for
    $d\geq 4$.
\end{abstract}

{\small\noindent\textbf{keywords:} Hilbert scheme of points, smoothability, Gorenstein
algebra, secant variety.\\
\textbf{MSC classes:} 14C05, 13H10, 14D15.}

\section{Introduction and notation}\label{sIntrNot}
Let $k$ be an algebraically closed field of characteristic neither $2$ nor $3$
and denote by $\Hilb_{p(t)}\mathbb{P}^N $ the Hilbert scheme parameterising closed subschemes in $\mathbb{P}^{N}$ with fixed Hilbert polynomial $p(t)\in{\Bbb Q}[t]$. Since A.~Grothendieck proved the existence of such a parameter space in 1966 (see \cite{Gro}), the problem of dealing with $\Hilb_{p(t)}\mathbb{P}^N $ and its subloci has been a fruitful field attracting the interest of many researchers in algebraic geometry.

Only to quickly mention some of the classical results which deserve, in our opinion, a particular attention, we recall Hartshorne's proof of the connectedness of $\Hilb_{p(t)}\mathbb{P}^N $ (see \cite{har66connectedness}), the description of the locus of codimension
$2$ arithmetically Cohen--Macaulay subschemes due to G.~Ellingsrud and
J.~Fogarty (see \cite{fogarty} for the dimension zero case and \cite{ellingsrud} for larger dimension) and of the study of the locus of codimension $3$ arithmetically Gorenstein subschemes due to J.~Kleppe and R.M.~Mir\'o--Roig (see \cite{roig_codimensionthreeGorenstein} and \cite{kleppe_roig_codimensionthreeGorenstein}).

If we restrict our attention to the case of zero--dimensional subschemes of
degree $d$, i.e. subschemes with Hilbert polynomial $p(t)=d$, then the first
significant results are due to J.~Fogarty (see \cite{fogarty}) and to A.~Iarrobino (see \cite{iarrobino_reducibility}).

In \cite{fogarty}, the author
proves that $\Hilb_{d} \mathbb{P}^{2}$ is smooth, hence irreducible thanks to Hartshorne's connectedness result (the same result holds, when one substitutes $\mathbb{P}^2$ by any smooth surface).

On the other hand in \cite{iarrobino_reducibility}, A.~Iarrobino
deals with the reducibility when $N\geq 3$ and $d$ is large with respect to $N$. In order to
better understand the result, recall that the locus of reduced schemes
${\mathcal R}\subseteq\Hilb_{d}\mathbb{P}^N $ is birational to a suitable open
subset of the $d$-th symmetric product of $\mathbb{P}^{N}$, thus it is
irreducible of dimension $dN$. We will denote by $\Hilb_{d}^{gen}\mathbb{P}^N
$ its closure in $\Hilb_{d}\mathbb{P}^N $. It is a well--known and easy fact
that $\Hilb_{d}^{gen}\mathbb{P}^N $ is an irreducible component of dimension
$dN$, by construction.
In \cite{iarrobino_reducibility}, the author proves that
$\Hilb_{d}\mathbb{P}^N $ is never irreducible when $d\gg N\ge 3$, showing that
there is a family of schemes
 of dimension greater than $dN$. Such a
 family is
thus necessarily contained in a component different from
    $\Hilb_{d}^{gen}\mathbb{P}^N$.

D.A.~Cartwright, D.~Erman, M.~Velasco,
B.~Viray proved that already for $d = 8$ and $N\geq 4$, the scheme $\Hilb_{d}\mathbb{P}^N$
is reducible (see \cite{CEVV}).

In view of these earlier works it seems reasonable to consider the
irreducibility and smoothness of open loci in $\Hilb_{d}\mathbb{P}^N $
defined by particular algebraic and geometric properties. In the present paper
we are interested in the locus $\HilbGor_{d}\mathbb{P}^N $ of points in
$\Hilb_{d}\mathbb{P}^N $ representing schemes which are Gorenstein. This is an
important locus: e.g.~it has an irreducible component
$\HilbGorGen_{d}\mathbb{P}^N :=\Hilb_{d}^{gen}\mathbb{P}^N \cap
\HilbGor_{d}\mathbb{P}^N $ of dimension $dN$ containing all the points
representing reduced schemes. Moreover it is open, but in general not dense,
inside $\Hilb_{d}\mathbb{P}^N$. Recently, interesting interactions between
$\HilbGor_{d}\mathbb{P}^N $ and the geometry of secant varieties and
general topology have been found (see
for example \cite{bubu2010}, \cite{Michalek}).

Some results about $\HilbGor_{d}\mathbb{P}^N $ are known. The
irreducibility and smoothness of $\HilbGor_{d}\mathbb{P}^N $ when $N\le 3$ is
part of the folklore (see \cite[Cor~2.6]{cn09} for more precise references). When $N\ge4$, the properties of $\HilbGor_{d}\mathbb{P}^N $ have been object of an intensive study in recent years.

E.g., it is classically known that $\HilbGor_{d}\mathbb{P}^N $ is never
irreducible for $d\ge 14$ and $N\ge6$, at least when the characteristic of $k$
is zero (see \cite{emsalem_iarrobino_small_tangent_space} and \cite{iakanev}: see also \cite{cn10}).
Also for $N = 4$ and $d\geq 140$ or $N = 5$ and $d\geq 42$ the scheme $\HilbGor_{d}\mathbb{P}^N$ is
reducible, see \cite[Section~6, p.~81]{bubu2010}. For fixed $N\in \{4, 5\}$ the
    minimal value of $d$, for which this scheme is reducible, is not known.

    As reflected by the quoted papers, it is natural to ask if
    $\HilbGor_{d}\mathbb{P}^N $ is irreducible when $d\le 13$ and $N$
    is arbitrary.
    There is some evidence of an affirmative answer to this question.
Indeed the first and third authors studied the  locus
$\HilbGor_{d}\mathbb{P}^N $ when $d\le 11$ and $N$ is arbitrary, proving its irreducibility and
dealing in detail with its singular locus in a series of papers \cite{cn09,
cn10, cn11, CN2stretched}.

A key point in the study of a zero--dimensional scheme
$X\subseteq\mathbb{P}^{N}$ is that it is abstractly isomorphic to $\Spec A$
where $A$ is an Artin $k$-algebra with $\dim_k(A)=d$. Moreover the
irreducible components of such an $X$ correspond bijectively to those direct
summands of $A$, which are local. Thus, in order to deal with
$\Hilb_{d}\mathbb{P}^N $, it suffices to deal with the irreducible schemes in
$\Hilb_{d'}\mathbb{P}^N $ for each $d'\le d$.

In all of the aforementioned papers, the methods used in the study of
$\HilbGor_{d}\mathbb{P}^N $ rely on an almost explicit classification of the
possible structure of local, Artin, Gorenstein $k$-algebras of length $d$.
Once such a classification is obtained, the authors prove that all the
corresponding irreducible schemes are smoothable, i.e.~actually lie in
$\HilbGorGen_{d}\mathbb{P}^N $. To this purpose they explicitly construct
a projective family flatly deforming the scheme they are interested
in (or, equivalently, the underlying algebra) to reducible schemes that they
know to be in $\HilbGorGen_{d}\mathbb{P}^N $ because their components have
lower degree.

Though such an approach sometimes seems to be too heavy in terms of calculations, only thanks to such a partial classification it is possible to state precise results about the singularities of $\HilbGor_{d}\mathbb{P}^N $.

However, in the papers \cite{cn10, cn11}, there
are families $H_d$ of schemes of degree $d$, where $d=10, 11$, for which an
explicit algebraic description in the above sense cannot be obtained (see
Section~3 of \cite{cn10} for the case $d=10$, Section~4 of \cite{cn11} for
$d=11$).
Nevertheless, using an alternative approach the authors are still able to
prove the irreducibility of $\HilbGor_{d}\mathbb{P}^N $ and study its
singular locus.  Indeed, using Macaulay's theory of inverse systems, the authors
check the irreducibility of the aforementioned loci $H_d$ inside $\HilbGor_{d}\mathbb{P}^N $. Then they show the existence of
a smooth point in $H_d\cap \HilbGorGen_{d}\mathbb{P}^N $. Hence, it follows
that $H_d\subseteq \HilbGorGen_{d}\mathbb{P}^N $.

The aim of the present paper is to refine and generalise this method. First, we
avoid a case by case approach by analysing large classes of algebras.
Second, in \cite{cn10, cn11} a direct check (e.g. using a computer algebra
program) is required to compute the dimension of tangent space to the Hilbert
scheme at some specific points to conclude that they are smooth. We avoid the need
of such computations by exhibiting classes of points which are smooth, making
the paper self--contained.

Using this method, we finally prove the following two statements.

\begin{mainthm}\label{ref:mainthm13degree}
If the characteristic of $k$ is neither $2$ nor $3$, then $\HilbGor_{d}\mathbb{P}^N $ is irreducible of dimension $dN$ for each $d\le 13$ and for $d=14$ and $N\le 5$.
\end{mainthm}

\begin{mainthm}\label{ref:mainthm14degree}
If the characteristic of $k$ is $0$ and $N\ge6$, then
$\HilbGor_{14}\mathbb{P}^N $ is connected and it has exactly two irreducible components,
which are generically smooth.
\end{mainthm}

Theorem~\ref{ref:mainthm13degree} has an interesting consequence regarding secant varieties of
Veronese embeddings. In \cite{geramita} Geramita conjectures that
the ideal of the $2^{nd}$ secant variety (the variety of secant lines) of the
$d^{th}$ Veronese embedding of $\mathbb{P}^{n}$ is generated by the $3\times3$ minors of
the $i^{th}$ catalecticant matrix for $2\le i\le d-2$. Such a conjecture was
confirmed in \cite{Raicu_thesis}. As pointed out in \cite[Section~8.1]{bubu2010}, the above
Theorem~\ref{ref:mainthm13degree} allows to extend the above result as follows: if $r\le 13$, $2r\le
d$ and, then for every $r \leq i\leq d-r$ the set--theoretic equations of the $r^{th}$
secant variety of the $d^{th}$ Veronese embedding of $\mathbb{P}^{n}$ are given by the
$(r+1)\times(r+1)$ minors of the $i^{th}$ catalecticant matrix.
\vspace{1em}

The proofs of Theorem~\ref{ref:mainthm13degree} and
Theorem~\ref{ref:mainthm14degree} are highly interlaced and they follow from a
long series of partial results. In order to better explain the ideas and
methods behind their proofs we will describe in the following lines the
structure of the paper.

In our analysis we incorporate several tools. In Section~\ref{sec:prelim} we
recall the classical ones, most notably Macaulay's correspondence for local,
Artinian, Gorenstein algebras and Macaulay's Growth Theorem. Moreover we also
list some criteria for checking the flatness of a family of algebras which will
be repeatedly used throughout the whole paper.

In Section~\ref{sec:dualgen} we analyse Artin Gorenstein quotients of a power series ring and
exploit the rich automorphism group of this ring to put the quotient into
suitable \emph{standard} form, deepening a result by A.~Iarrobino.

In Section~\ref{sec:specialforms} we further
analyse the quotients, especially their dual socle
generators. We also
construct several irreducible subloci of the Hilbert scheme using the theory
of secant varieties. We give a small contribution to this theory, showing that
the fourth secant variety to a Veronese reembedding of $\mathbb{P}^n$ is
defined by minors of a suitable catalecticant matrix.

Section~\ref{sec:ray} introduces a central object in our study: a class of
families, called ray families, for which we have relatively good control of
the flatness and, in special cases, fibers. Most
notably, Subsection~\ref{subsec:tangentpreserving} gives a class of
\emph{tangent preserving} flat families, which enable us to construct smooth
points on the Hilbert scheme of points without the necessity of heavy computations.

Finally, in Section~\ref{sec:proof}, we give the proofs of
Theorem~\ref{ref:mainthm13degree} and ~\ref{ref:mainthm14degree}. It is worth
mentioning that these results are rather easy consequences of the introduced
machinery. In this section we also prove the following general smoothability
result (see Thm~\ref{ref:mainthmstretchedfive:thm}), which has no restriction
on the length of the algebra and generalises the smoothability results
    from
    \cite{SallyStretchedGorenstein}, \cite{CN2stretched} and
    \cite{EliasVallaAlmostStretched}.
\begin{mainthm}
    Let $k$ be an algebraically closed field of characteristic neither $2$ nor
    $3$.
    Let $A$ be a local Artin Gorenstein $k$-algebra with maximal ideal
    $\mathfrak{m}$.

    If $\dim_k(\mathfrak{m}^2/\mathfrak{m}^3)\le 5$ and
    $\dim_k(\mathfrak{m}^3/\mathfrak{m}^4)\le 2$, then $\Spec A$ is smoothable.
\end{mainthm}

\subsection*{Notation}
\DDef{P}{P}%
\DDef{S}{S}%
\DDef{mmS}{\mathfrak{m}_{\DS}}%
\def\Dx{\alpha}%
\def\Dan#1{\annn{\DS}{#1}}%
\def\Dhdvect#1{\Delta_{#1}}%

All symbols appearing below are defined in Section~\ref{sec:prelim}.

\noindent\begin{longtable}{p{3.5cm} p{12cm}}
$k$ & an algebraically closed field of characteristic $\neq 2, 3$.\\
$\DP = k[x_1, \ldots ,x_n]$ & a polynomial ring in $n$ variables and fixed
basis.\\
$\DS = k[[\Dx_1, \ldots ,\Dx_n]]$ & a power series ring dual
(see~Subsection \ref{sss:apolarity}) to $\DP$, with a fixed (dual) basis.\\
$\DmmS$ & the maximal ideal of $\DS$.\\
$\DS_{poly} = k[\Dx_1, \ldots ,\Dx_n]$ & a polynomial subring of
$\DS$ defined by the choice of the basis.\\
$H_A$ & the Hilbert function of a local Artin algebra $A$.\\
$\Dhdvect{A, i}$, $\Dhdvect{i}$ & the $i$-th row of the symmetric decomposition of the Hilbert function of
a local Artin Gorenstein algebra $A$ as in Theorem~\ref{ref:Hfdecomposition:thm}.\\
$e(a)$ & the $a$-th ``embedding dimension'', equal to $\sum_{t=0}^a \Dhdvect{t}(1)$, as in Definition~\ref{ref:standardform:def}.\\
$\Dan{f}$ & the annihilator of $f\in \DP$ with respect to the action of $\DS$.\\
$\Apolar{f}$ & the apolar algebra of $f\in \DP$, equal to $\DS/\Dan{f}$.
\end{longtable}

\section{Preliminaries}\label{sec:prelim}

\def\DSpoly{S_{poly}}%
Let $n$ be a natural number. By $(\DS, \DmmS, k)$ we denote the power series ring
$k[[\Dx_1, \ldots ,\Dx_n]]$ of dimension $n$ with a fixed basis $\Dx_1, \ldots
,\Dx_n$. The chosen basis determines a polynomial ring $\DSpoly =
k[\Dx_1, \ldots ,\Dx_n] \subseteq \DS$. By $\DP$ we denote the polynomial ring
$k[x_1, \ldots ,x_n]$. We will later define a duality between $\DS$ and $\DP$,
see Subsection~\ref{sss:apolarity}. We usually think of $n$ being large enough, so that the
considered local Artin algebras are quotients of $\DS$.

For an element $f\in \DP$, we say that $f$
\emph{does not contain $x_i$} if $f\in k[x_1, \ldots , x_{i-1}, x_{i+1}, \ldots
,x_n]$; similarly for $\sigma\in \DS$ or $\sigma\in \DSpoly$. For $f\in \DP$,
by $f_d$ we denote the degree $d$ part of $f$, with respect to the total
degree; similarly for $\sigma\in \DS$.

By $\DP_{m}$ and $\DP_{\leq m}$ we denote the space of homogeneous polynomials of
degree $m$ and (not necessarily homogeneous) polynomials of degree at most $m$ respectively. These spaces are
naturally affine spaces over $k$, which equips them with a scheme structure.

Recall that $\DS$ has a rich automorphism group: for every choice of elements
    $\sigma_1, \ldots ,\sigma_n\in \DmmS$ linearly independent in
    $\DmmS/\DmmS^2$ there is a unique automorphism $\varphi$ of $\DS$ such that
    $\varphi(\Dx_i) = \sigma_i$ for $i=1,2, \ldots ,n$. The existence of such
    automorphisms is employed in Section~\ref{sec:specialforms} to put the
    considered Artin Gorenstein algebras in a better form. See
    e.g.~\cite[Section~2]{EliasRossi} for details and examples of this method.

\begin{remark}
    For the reader's convenience we introduce numerous examples, which
    illustrate the possible applications. In all these examples $k$ may have
    arbitrary characteristic $\neq 2, 3$ unless otherwise stated.
\end{remark}

\subsection{Artin Gorenstein schemes and algebras}

    In this section we recall the basic facts about Artin Gorenstein algebras.
    For a more thorough treatment we refer to \cite{iakanev},
    \cite{EisView}, \cite{cn09}
    and \cite{JelMSc}.

    Finite type zero-dimensional schemes correspond to Artin algebras. Every
    such algebra $A$ splits as a finite product of its localisations at maximal
    ideals, which corresponds to the fact that the support of $\Spec A$ is
    finite and totally disconnected.
    Therefore, we will focus our interest on \emph{local} Artin
    $k$-algebras. Since $k$ is algebraically closed, such algebras
    have residue field $k$.

    An important invariant of a local algebra $(A, \mathfrak{m}, k)$ is its Hilbert function $H_A$ defined by
    $H_A(l) = \dimk \mathfrak{m}^l/\mathfrak{m}^{l+1}$. Since $H_A(l) = 0$ for
    $l \gg 0$ it is usual to write $H_A$ as the vector of its non-zero values.
    The \emph{socle degree} of $A$ is the largest $l$ such that $H_A(l) \neq
    0$.
    Such an algebra is
    \emph{Gorenstein} if the annihilator of $\mathfrak{m}$ is a
    one-dimensional vector space over $k$, see \cite[Chap~21]{EisView}.

    We
    recall for reader's benefit that a finite not necessarily local algebra $A$ is Gorenstein if and
    only if all its localisations at maximal ideals are Gorenstein (in
    particular it is meaningful to discuss the irreducibility of the
    Gorenstein locus in the Hilbert scheme by reducing to the study of
    deformations of local Gorenstein algebras: see
    Section~\ref{sss:smoothability}).

    Since $k$ is algebraically closed, we may write
    each Artin local algebra $(A, \mathfrak{m}, k)$ as a quotient of the power series ring $\DS =
    k[[\Dx_1, \ldots ,\Dx_n]]$ when $n$ is large enough, in fact $n\geq
    H_A(1)$ is sufficient. Since $\dimk A$ is
    finite, such a presentation gives a presentation $A = \DSpoly/I$,
    i.e.~a point $[\Spec A]$ of the Hilbert scheme of $\mathbb{A}^n = \Spec
    \DSpoly$.

\subsection{Contraction map and apolar algebras}\label{sss:apolarity}

\DDef{aa}{\mathbf{a}}
\DDef{bb}{\mathbf{b}}
In this section we introduce the contraction mapping, which is closely related
to Macaulay's inverse systems. We refer to \cite{ia94} and
\cite[Chap~21]{EisView} for details and proofs.

Recall that $\DP = k[x_1, \ldots ,x_n]$ is a polynomial ring and $\DS =
k[[\Dx_1, \ldots ,\Dx_n]]$ is a power series ring. The $k$-algebra $\DS$ acts
on $\DP$ by \emph{contraction} (see \cite[Def~1.1]{iakanev}). This action is denoted by
$(\cdot)\hook(\cdot):\DS \times \DP \to \DP$ and defined as follows.
Let
$\mathbf{x}^{\Daa} = x_1^{a_1} \ldots x_{n}^{a_n}\in \DP$ and $\mathbf{\Dx}^{\Dbb} = \Dx_1^{b_1} \ldots \Dx_{n}^{b_n}\in
\DS$ be monomials. We write $\Daa \geq \Dbb$ if and only if $a_i\geq b_i$ for all
$1\leq i\leq n$. Then
\[
    \mathbf{\Dx}^{\Dbb}\hook \mathbf{x}^{\Daa} :=
    \begin{cases}
        \mathbf{x}^{\Daa - \Dbb} & \mbox{if}\quad
        \Daa\geq\Dbb\\
        0 & \mbox{otherwise.}
\end{cases}\]
This action extends to $\DS\times \DP \to \DP$ by $k$-linearity on $\DP$ and
countable $k$-linearity on $\DS$.

The contraction action induces a perfect pairing between
$\DS/\DmmS^{s+1}$ and $\DP_{\leq s}$, which restricts to a perfect pairing between the degree $s$
polynomials in $\DSpoly$ and $\DP$. These pairings are compatible for different
choices of $s$.

If $f\in \DP$ then a \emph{derivative} of $f$ is an element of the
$\DS$-module $\DS f$, i.e.~an
element of the form $\partial\hook f$ for $\partial\in \DS$. By definition, these elements form an
$\DS$-submodule of $\DP$, in particular a $k$-linear subspace.

Let $A = \DS/I$ be an Artin quotient of $\DS$, then $A$ is local. The contraction
action associates to $A$ an $\DS$-submodule $M \subseteq \DP$ consisting of elements
annihilated by $I$, so that $A$ and $M$ are dual. If $A$ is Gorenstein, then the $\DS$-module
$M$ is cyclic, generated by a polynomial $f$ of degree $s$ equal to the socle
degree of $A$.
We call every such $f$ a \emph{dual socle generator of the Artin
Gorenstein algebra $A$}.
Unlike $M$, the polynomial $f$ is \emph{not determined uniquely
    by the choice of the presentation $A = S/I$}, however if $f$ and $g$ are two dual socle generators,
then $g = \partial\hook f$, where $\partial\in \DS$ is invertible.

Conversely, let $f\in \DP$ be a polynomial of degree $s$. We can associate
it the ideal $I := \Dan{f}$ such that $A := \DS/I$ is
a local Artin Gorenstein algebra of socle degree $s$. We call $I$ the
\emph{apolar ideal} of $f$ and $A$ the
\emph{apolar algebra} of $f$, which we denote as
\[A = \Apolar{f}.\] From the discussion above it follows that every
local Artin Gorenstein algebra is an apolar algebra of some polynomial.

\begin{remark}\label{ref:dualautomorphisms:remark}
    Recall that we may think of $\DS/\DmmS^{s+1}$ as the linear space dual
    to $\DP_{\leq s}$. An automorphism $\psi$ of $\DS$ or
    $\DS/\DmmS^{s+1}$ induces an automorphism
    $\DPut{psistar}{\psi^*}$ of the
    $k$-linear space $\DP_{\leq s}$.
    If $f\in \DP_{\leq s}$ and $I$ is
    the apolar ideal of $f$, then $\psi(I)$ is the apolar ideal of
    $\psi^*(f)$. Moreover, $f$ and $\psi^*(f)$ have the same degree.
\end{remark}

\subsection{Iarrobino's symmetric decomposition of Hilbert function}

\def\Dhd#1#2{\Delta_{#1}\pp{#2}}
\def\iatf#1#2{(\DS f)_{#1}^{#2}}
\def\Ddegf{s}

One of the most important invariants possessed by a local Artin Gorenstein algebra is
the symmetric decomposition of its Hilbert function, due to Iarrobino~\cite{ia94}.
To state the theorem it is convenient to define addition of vectors of
different lengths position-wise: if $a = (a_0, \ldots ,a_n)$ and $b = (b_0,
\ldots ,b_m)$ are vectors, then $a+ b= (a_0 + b_0,  \ldots , a_{\max(m, n)} +
b_{\max(m, n)})$, where $a_i = 0$ for $i >n $ and $b_i = 0$ for $i > m$.
In the following, all vectors are indexed starting from zero.

{\def\Dmm{\mathfrak{m}}%
\def\Dmmperp#1{(0 : \Dmm^{#1})}%
Let $(A, \Dmm, k)$ be a local Artin Gorenstein algebra. By
    $\Dmmperp{l}$ we
denote the annihilator of $\Dmm^{l}$ in $A$. The chain $0 = \Dmmperp{0} \subseteq
\Dmmperp{1} \subseteq  \ldots $ defines a filtration on $A$. In
general, it is different from
the usual filtration $0 = \Dmm^{s+1} \subseteq \Dmm^s \subseteq \Dmm^{s-1} \subseteq
\ldots $. The analysis of mutual position of these filtrations is the content
of Theorem~\ref{ref:Hfdecomposition:thm} below.}

\begin{thm}[Iarrobino's symmetric decomposition of Hilbert function]\label{ref:Hfdecomposition:thm}
    \def\Dmm{\mathfrak{m}}%
    \def\Dmmperp#1{(0: \mathfrak{m}^{#1})}%
    Let $(A, \Dmm, k)$ be a local Artin Gorenstein algebra of socle degree $s$ and Hilbert
    function $H_A$.
    Let
    \[
        \Dhdvect{i}(t) := \dimk \frac{\Dmmperp{s+1-i-t}\cap
        \Dmm^{t}}{\Dmmperp{s-i-t}\cap \Dmm^{t} + \Dmmperp{s+1-i-t} \cap
        \Dmm^{t+1}}\quad \mbox{for}\quad t=0,1, \ldots , s-i.
    \]
    The vectors $\Dhdvect{0}, \Dhdvect{1},
    \ldots ,\Dhdvect{s}$ have the following properties:
    \begin{enumerate}
        \item the vector $\Dhdvect{i}$ has length $s + 1 - i$ and satisfies
            $\Dhdvect{i}(t) = \Dhdvect{i}(s - i - t)$ for all
            integers $t\in [0,
            s-i]$.
        \item the Hilbert function $H_A$ is equal to the sum
            $\sum_{i=0}^{s} \Dhdvect{i}$.
        \item the vector $\Dhdvect{0}$ is equal to the Hilbert function of a
            local Artin Gorenstein \emph{graded} algebra of socle degree $s$.
    \end{enumerate}
\end{thm}

Let $(A, \mathfrak{m}, k)$ be a local Artin Gorenstein algebra.
There are a few important remarks to do.
\begin{enumerate}
    \item Since $\Dhdvect{0}$ is the Hilbert function of an algebra, we have $\Dhdvect{0}(0) =
        1 = H_A(0)$. Thus for every $i > 0$ we have $\Dhdvect{i}(0) = 0$. From
        symmetry it follows that $\Dhdvect{i}(s+1-i) = 0$. In particular
        $\Dhdvect{s} = (0)$ and $\Dhdvect{s-1} = (0, 0)$, so we may ignore
        these vectors. On the other hand $\Dhdvect{s-2} = (0, q, 0)$ is in
        general non-zero and its importance is illustrated by
        Proposition~\ref{ref:squares:prop}.
    \item Suppose that $H_A = (1, n, 1, 1)$ for some $n > 0$. Then we have
        $\Dhdvect{0} = (1, *, *, 1)$ and $\Dhdvect{1} = (0, *, 0)$, thus
        $\Dhdvect{0} = (1, *, 1, 1)$, so that $\Dhdvect{0} = (1, 1, 1, 1)$
        because of its symmetry. Then $\Dhdvect{1} = (0, n -1, 0)$. Similarly,
        if $H_A = (1, n, e, 1)$ is the Hilbert
        function of a local Artin Gorenstein algebra, then $n\geq e$. This is
        a basic example on how Theorem~\ref{ref:Hfdecomposition:thm} imposes
        restrictions on the Hilbert function of $A$.
    \item If $A$ is graded, then $\Dhdvect{0} = H_A$ and all other
        $\Dhdvect{\bullet}$ are zero vectors, see \cite[Prop~1.7]{ia94}.
    \item For every $a\leq s$ the partial sum $\sum_{i=0}^a \Dhdvect{i}$ is the Hilbert
        function of a local Artin graded algebra, see \cite[Def~1.3,
        Thm~1.5]{ia94}, see also \cite[Subsection~1.F]{ia94}. In particular it satisfies Macaulay's Growth
        Theorem, see Subsection~\ref{ref:MacGrowth:sss}. Thus e.g. there is no
        local Artin Gorenstein algebra with Hilbert function decomposition
        satisfying $\Dhdvect{0} = (1, 1, 1, 1, 1, 1)$ and $\Dhdvect{1} = (0, 0, 1, 0, 0)$,
        because then $(\Dhdvect{0} + \Dhdvect{1})(1) = 1$ and
        $(\Dhdvect{0} + \Dhdvect{1})(2) = 2$.
\end{enumerate}

Let us now analyse the case when $A = \Apolar{f} = \DS/\Dan{f}$ is the apolar algebra of a
polynomial $f\in \DP$, where $f =
\sum_{i=0}^s f_i$ for some $f_i\in \DP_i$. Each local Artin Gorenstein algebra
is isomorphic to such algebra, see~Subsection~\ref{sss:apolarity}. For the
proofs of the following remarks, see~\cite{ia94}.

\begin{enumerate}
    \item Vector $\Dhdvect{0}$ is equal to the Hilbert function of $\Apolar{f_s}$, the apolar algebra of the leading form of~$f$.
    \item If $A$ is graded, then $\Dan{f} = \Dan{f_s}$, so that we may always
        assume that $f = f_s$. Moreover, in this case $H_A(m)$ is equal to
        $(\DS f_s)_m$, the number of degree $m$ derivatives of $f_s$.
    \item Let $f_1$, $f_2$ be polynomials of degree $s$ such that $f_1 - f_2$ is a
        polynomial of degree $d < s$. Let $A_i = \Apolar{f_i}$ and let
        $\Dhdvect{A_i, m}$ be the symmetric decomposition of the Hilbert
        function $H_{A_i}$ of $A_i$ for $i=1,2$. Then $\Dhdvect{A_1, m} =
        \Dhdvect{A_2, m}$ for all $m < s - d$, see
        \cite[Lem~1.10]{ia94}.
\end{enumerate}

\subsection{Smoothability and unobstructedness}\label{sss:smoothability}

An Artin algebra $A$ is called \emph{smoothable} if it is a (finite flat) limit of
smooth algebras, i.e.~if there exists a finite flat family over an irreducible
base with a special fiber isomorphic to $\Spec A$ and general fiber smooth.
Recall that $A  \simeq A_{\mathfrak{m}_1} \times  \ldots
A_{\mathfrak{m}_r}$, where $\mathfrak{m}_i$ are maximal ideals of $A$. The algebra
$A$ is smoothable if all localisations $A_{\mathfrak{m}}$ at its maximal ideals
are smoothable. The converse also holds, i.e.~if
    an algebra $A \simeq B_1 \times B_2$ is smoothable, then the algebras
    $B_1$ and $B_2$ are also smoothable, a complete and characteristic free proof of this fact will appear
    shortly in \cite{jabu_jelisiejew_smoothability}.
    We say that a zero-dimensional scheme $Z = \Spec A$
    is \emph{smoothable} if the algebra $A$ is
smoothable.

It is crucial that every local Artin Gorenstein algebra $A$ with $H_{A}(1)
\leq 3$ is smoothable, see~\cite[Prop~2.5]{cn09}, which follows from the
Buchsbaum-Eisenbud classification of resolutions,
see~\cite{BuchsbaumEisenbudCodimThree}.
Also complete intersections are smoothable.
    A complete intersection $Z \subseteq \mathbb{P}^n$ is smoothable by
Bertini's Theorem (see \cite[Example 29.0.1]{HarDeform}, but note that
Hartshorne uses a slightly weaker definition of smoothability, without
finiteness assumption). If $Z = \Spec A$ is a complete
intersection in $\mathbb{A}^n$, then $Z$ is a union of connected components
of a complete intersection $Z' = \Spec B$ in $\mathbb{P}^n$, so that $B
\simeq A \times C$ for some algebra $C$. The algebra
$B$ is smoothable since $Z'$ is. Thus also the algebra $A$ is
smoothable, i.e. $Z$ is smoothable.

\def\Hilb#1#2{\mathcal{H}ilb_{#2}(\mathbb{P}^{#1})}
\begin{defn}\label{ref:nonobstructed:def}
    A smoothable Artin algebra $A$ of length $d$, corresponding to $\Spec A
    \subseteq \mathbb{P}^n$, is \emph{unobstructed} if
    the tangent space to $\Hilb{n}{d}$ at the $k$-point $[\Spec A]=:p$ has dimension $nd$. If $A$ is
    unobstructed, then $p$ is a smooth point of the Hilbert scheme.
\end{defn}

The unobstructedness is independent of $n$ and the chosen embedding of $\Spec A$ into
$\mathbb{P}^n$, see discussion before \cite[Lem 2.3]{cn09}.
The argument above shows that algebras corresponding to complete intersections in $\mathbb{A}^n$ and
$\mathbb{P}^n$ are unobstructed. Every local Artin Gorenstein algebra $A$
with $H_A(1)\leq 3$ is unobstructed, see~\cite[Prop~2.5]{cn09}. Moreover,
every local Artin Gorenstein algebra $A$ with $H_A(1) \leq 2$ is a complete
intersection in $\mathbb{A}^2$ by the Hilbert-Burch theorem.

\begin{defn}\label{ref:limitreducible:def}
    An Artin algebra $A$ is \emph{limit-reducible} if there exists
    a flat family (over an irreducible base) whose special fiber is $A$ and general fiber is reducible.
    An Artin algebra $A$ is \emph{strongly non-smoothable} if it is not
    limit-reducible.
\end{defn}
Clearly, strongly non-smoothable algebras (other than $A = k$) are non-smoothable.
The definition of strong non-smoothability is useful, because to
show that there is no non-smoothable algebra of length less than $d$ it is
enough to show that there is no strongly non-smoothable algebra of
length less than $d$.

\subsection{Macaulay's Growth Theorem}\label{ref:MacGrowth:sss}

We will recall
Macaulay's Growth Theorem and Gotzmann's Persistence Theorem, which provide strong restrictions on the possible
Hilbert functions of graded algebras.
Fix $n\geq 1$. Let $m$ be any natural number, then $m$ may be uniquely written
in the form
\[m =\binom{m_n}{n} + \binom{m_{n-1}}{n-1} +  \ldots + \binom{m_1}{1},\] where
$m_n > m_{n-1} >  \ldots  > m_1$. We define
\[
    m^{\langle i\rangle} := \binom{m_n+1}{n+1} + \binom{m_{n-1}+1}{n} +  \ldots
    + \binom{m_1+1}{2}.
\]

It is useful to compute some initial values of the above defined function, i.e.
$1^{\langle n\rangle} = 1$ for all $n$, $3^{\langle 2 \rangle} = 4$,
$4^{\langle 2\rangle} = 5$, $6^{\langle 2\rangle} = 10$ or $4^{\langle
3\rangle} = 5$.

\begin{thm}[Macaulay's Growth Theorem]\label{ref:MacaulayGrowth:thm}
    If $A$ is a graded quotient of a polynomial ring over $k$, then the
    Hilbert function $H_A$ of $A$ satisfies $H_A(m+1) \leq H_{A}(m)^{\langle
    m\rangle}$ for all $m$.
\end{thm}
\begin{proof}
    See \cite[Thm 4.2.10]{BrunsHerzog}.
\end{proof}

Note that the assumptions of Theorem~\ref{ref:MacaulayGrowth:thm} are satisfied
for every local Artin $k$-algebra $(A, \mathfrak{m}, k)$, since its Hilbert
function is by definition equal to the
Hilbert function of the associated graded algebra.

\begin{remark}\label{ref:MacaulaysBoundedGrowth:rmk}
    We will frequently use the following easy consequence of Theorem~\ref{ref:MacaulayGrowth:thm}.

    Let $A$ be a graded quotient of a polynomial ring over $k$.
    Suppose that  $H_A(l) \leq l$ for some $l$.
    Then $H_A(l) = \binom{l}{l} + \binom{l-1}{l-1} +  \ldots $ and
    $H_A(l)^{\langle l \rangle} = \binom{l+1}{l+1} + \binom{l}{l} +  \ldots  =
    H_A(l)$, thus $H_A(l+1) \leq H_A(l)$. It follows
        that the Hilbert function of $H_A$ satisfies $H_A(l) \geq H_A(l+1)\geq
        H_A(l+2) \geq  \ldots $. In particular $H_A(m) \leq l$ for all $m \geq
        l$.
\end{remark}

\begin{thm}[Gotzmann's persistence Theorem]\label{ref:Gotzmann:thm}
    Let $A = \DSpoly/I$ be a graded quotient of a polynomial ring $\DSpoly$ over $k$ and
    suppose that for some $l$ we have $H_A(l+1) = H_{A}(l)^{\langle
    l\rangle}$ and $I$ is generated by elements of degree at most $l$. Then
    $H_A(m+1) = H_{A}(m)^{\langle m\rangle}$ for all $m\geq l$.
\end{thm}
\begin{proof}
    See \cite[Thm 4.3.3]{BrunsHerzog}.
\end{proof}

In the following we will mostly use the following consequence of
Theorem~\ref{ref:Gotzmann:thm}, for which we introduce some (non-standard) notation. Let
$I \subseteq \DSpoly = k[\Dx_1, \ldots ,\Dx_n]$ be a graded ideal in a polynomial ring and
$m\geq 0$. We say that $I$ is \emph{$m$-saturated} if for all $l\leq m$ and
$\sigma\in (\DSpoly)_l$ the condition  $\sigma\cdot (\Dx_1, \ldots ,\Dx_n)^{m - l}\subseteq I$
implies $\sigma\in I$.

\begin{lem}\label{ref:P1gotzmann:lem}
    \def\Dnn{\mathfrak{n}}%
    Let $\DSpoly = k[\Dx_1, \ldots ,\Dx_n]$ be a polynomial ring with
    maximal ideal $\Dnn = (\Dx_1,
    \ldots ,\Dx_n)$. Let $I \subseteq \DSpoly$ be a graded ideal and $A =
    \DSpoly/I$. Suppose that $I$ is $m$-saturated for some
    $m\geq 2$.
    Then
    \begin{enumerate}
        \item if $H_A(m) = m+1$ and $H_A(m+1) = m+2$, then $H_A(l) = l+1$ for
            all $l\leq m$, in particular $H_A(1) = 2$.
        \item if $H_A(m) = m+2$ and $H_A(m+1) = m+3$, then $H_A(l) = l+2$ for
            all $l\leq m$, in particular $H_A(1) = 3$.
    \end{enumerate}
\end{lem}

\begin{proof}
    1. First, if $H_A(l) \leq l$ for some $l < m$, then by Macaulay's Growth
    Theorem $H_A(m)\leq l < m+1$, a contradiction. So it suffices to prove
    that $H_A(l) \leq l+1$ for all $l < m$.

    Let $J$ be the ideal generated by elements of degree at most $m$ in $I$.
    We will prove that the graded ideal $J$ of $\DSpoly$  defines a
    $\mathbb{P}^1$ linearly embedded into $\mathbb{P}^{n-1}$.

    Let $B = \DSpoly/J$. Then $H_B(m) = m+1$ and $H_B(m+1) \geq m+2$. Since $H_B(m)
    = m+1 =
    \binom{m+1}{m}$, we have $H_B(m)^{\langle m\rangle} = \binom{m+2}{m+1} =
    m+2$ and by Theorem \ref{ref:MacaulayGrowth:thm} we get $H_B(m+1)\leq m+2$, thus $H_B(m+1) = m+2$. Then by
    Gotzmann's
    Persistence Theorem $H_B(l) = l+1$ for all $l > m$. This implies that the
    Hilbert polynomial of $\Proj B \subseteq \mathbb{P}^{n-1}$ is $h_B(t) = t+1$, so
    that $\Proj B \subseteq \mathbb{P}^{n-1}$ is a linearly embedded $\mathbb{P}^1$. In
    particular the Hilbert function and Hilbert polynomial of $\Proj B$ are
    equal
    for all arguments.
    By assumption, we have $J_l = J^{sat}_l$ for all $l < m$. Then  $H_{A}(l)
     = H_{\DSpoly/J}(l) = H_{\DSpoly/J^{sat}}(l) = l+1$ for all $l < m$ and the claim of
    the lemma follows.

    2. The proof is similar to the above one; we mention only the points,
    where it changes. Let $J$ be the ideal generated
    by elements of degree at most $m$ in $I$ and $B = \DSpoly/J$. Then $H_B(m) = m
    +2 = \binom{m+1}{m} + \binom{m-1}{m-1}$, thus $H_B(m+1) \leq
    \binom{m+2}{m+1} +
    \binom{m}{m} = m+3$ and $B$ defines a closed subscheme of $\mathbb{P}^{n-1}$ with
    Hilbert polynomial $h_B(t) = t+2$. There are two isomorphism types of such
    subschemes: $\mathbb{P}^1$ union a point and $\mathbb{P}^1$ with an
    embedded double point. One checks that for these schemes the Hilbert
    polynomial is equal to the Hilbert function for all arguments and then
    proceeds as in the proof of Point 1.
\end{proof}

\begin{remark}\label{ref:GorensteinSaturated:rmk}
    \def\Dnn{\mathfrak{n}}%
    If $A = \DSpoly/I$ is a graded Artin Gorenstein algebra of socle degree $s$, then
    it is $m$-saturated for every $m\leq s$.
    Indeed, we may assume that $A = \Apolar{F}$ for some homogeneous $F\in
    \DP$ of degree $s$, then $I = \Dan{F}$. Let $\Dnn = (\Dx_1, \ldots ,\Dx_n) \subseteq
    k[\Dx_1, \ldots ,\Dx_n] = \DSpoly$.
    Take $\sigma\in (\DSpoly)_l$, then $\sigma\in I$ if and only if $\sigma\hook F =
    0$. Similarly, $\sigma\Dnn^{m-l} \subseteq I$ if and only if every element
    of $\Dnn^{m-l}$ annihilates $\sigma\hook F$. Since $\sigma\hook F$ is
    either a homogeneous polynomial of degree $s - l \geq m -l$ or it is zero, both conditions are
    equivalent.
\end{remark}

\begin{remark}
    Clearly, if two graded ideals $I$ and $J$ of $\DSpoly$ agree up to degree $m$
        and $I$ is \hbox{$m$-saturated}, then also $J$ is $m$-saturated.
\end{remark}

\subsection{Flatness over $\Spec k[t]$}

For further reference we explicitly state a purely elementary flatness
criterion. Its formulation is a bit complicated, but this is precisely the
form which is needed for the proofs. This criterion relies on the easy
observation that the torsion-free modules over $k[t]$ are flat.
\begin{prop}\label{ref:flatelementary:prop}
    Suppose $S$ is a $k$-module and $I \subseteq S[t]$ is a $k[t]$-submodule. Let $I_0 := I\cap S$.
    If for every $\lambda\in k$ we have
    \[(t-\lambda)\cap I \subseteq (t-\lambda)I + I_0[t],\] then $S[t]/I$ is a flat $k[t]$-module.
\end{prop}

\begin{proof}

The ring $k[t]$ is a principal ideal domain, thus a $k[t]$-module is flat if and only
if it is torsion-free, see \cite[Cor~6.3]{EisView}.
Since every polynomial in $k[t]$ decomposes into linear factors, to prove that $M = S[t]/I$ is
torsion-free it is enough to show that $t-\lambda$ are non--zerodivisors on
$M$,~i.e.~that $(t-\lambda)x\in I$ implies $x\in I$ for all $x\in S[t]$,
$\lambda\in k$.

    Fix $\lambda\in k$ and suppose that $x\in
    S[t]$ is such that $(t-\lambda)x \in I$. Then by assumption $(t-\lambda)x\in
    (t-\lambda)I + I_0[t]$, so that $(t-\lambda)(x-i) \in I_0[t]$ for some $i\in I$.
    Since $S[t]/I_0[t]   \simeq S/I_0[t]$ is a free $k[t]$-module, we have
    $x-i\in I_0[t]\subseteq I$ and so $x\in I$.
\end{proof}

\begin{remark}\label{ref:flatnessremovet:remark}
    Let $i_1, \ldots ,i_r$ be the generators of $I$. To check the inclusion
    which is the assumption
    of Proposition~\ref{ref:flatelementary:prop}, it is enough to
    check that $s\in (t-\lambda)\cap I$ implies $s\in (t-\lambda)I + I_0[t]$
    for all $s
    = s_1 i_1 +  \ldots  + s_r i_r$, \emph{where $s_i\in S$}.

    Indeed, take an arbitrary element $s\in I$ and write $s = t_1 i_1 +  \ldots  + t_r i_r$, where $t_1, \ldots ,t_r\in
    S[t]$. Dividing $t_i$ by $t-\lambda$ we obtain $s = s_1 i_1 +  \ldots  +
    s_r i_r + (t-\lambda)i$, where $i\in I$ and $s_i\in S$. Denote $s' = s_1
    i_1 +  \ldots  + s_r i_r$, then
    $s\in (t-\lambda)\cap I$ if and only if $s'\in (t-\lambda)\cap I$ and $s\in (t-\lambda)I +
    I_0[t]$ if and only if $s'\in (t-\lambda)I + I_0[t]$.
\end{remark}

\begin{example}\label{ref:exampleflatcrit:example}
    Consider $S = k[x,y]$ and $I = xyS[t] + (x^3 - tx)S[t] \subseteq S[t]$. Take an element $s_1 xy + s_2(x^3 - tx)\in
    I$ and suppose $s_1 xy + s_2(x^3 - tx)\in (t-\lambda)S[t]$. We want to prove
that this element lies in $I_0[t] + (t-\lambda)I$.
As in Remark~\ref{ref:flatnessremovet:remark}, by subtracting an element of $I(t-\lambda)$ we may assume that $s_1, s_2$ lie in $S$.
Then $s_1 xy + s_2(x^3 - tx)\in (t-\lambda)S[t]$ if
and only if $s_1 xy + s_2(x^3 - \lambda x) = 0$. In particular we have $s_2
\in yS$ so that $s_2 (x^3 - tx)\in xyS[t]$, then $s_1 xy + s_2(x^3 - tx)\in xyS[t] \subseteq I_0[t]$.
\end{example}

Similarly as in Example~\ref{ref:exampleflatcrit:example}, in the
following we will frequently use the following easy observation, which we
state in Lemma~\ref{ref:decompositionhomog:lem}.

\begin{lem}\label{ref:decompositionhomog:lem}
    Consider
    a ring $R = B[\Dx]$ graded by the degree of $\Dx$.
    Let $d$ be a natural number and $I\subseteq R$ be a homogeneous ideal
    generated in degrees less or equal to $d$.

    Let $q\in B[\Dx]$ be an element of $\alpha$-degree strictly
    less than $d$ and such that for every $b\in B$ satisfying $b\Dx^{d}\in I$,
    we have $bq\in I$.
    Then for every $r\in R$ the condition
    \[r(\Dx^d - q)\in I\ \ \mbox{implies}\ \
        r\Dx^{d}
    \in I \ \ \mbox{and}\ \ rq\in I.\]
\end{lem}
\begin{proof}
    We apply induction with respect to $\alpha$-degree of $r$, the base
    case being $r = 0$.
    Write
    \[r = \sum_{i=0}^{m} r_i \Dx^i,\quad\mbox{where}\quad r_i\in B.\]
    The leading form of $r(\Dx^d - q)$ is $r_m\Dx^{m+d}$ and it lies in $I$. Since $I$
    is homogeneous and generated in degree at most $d$, we have $r_m \Dx^d\in I$. Then $r_m
    q\in I$ by assumption, so that $\hat{r} := r - r_m\Dx^{m}$ satisfies
    $\hat{r}(\Dx^{d} - q)\in I$. By induction we have
    $\hat{r}\Dx^d,\,\hat{r}q \in I$, then also $r\Dx^d,\,rq\in I$.
\end{proof}

\section{Standard form of the dual generator}\label{sec:dualgen}

    \begin{defn}\label{ref:standardform:def}
        Let $f\in \DP = k[x_1,\dots,x_n]$ be a polynomial of degree $s$. Let $I = \annn{\DS}{f}$
        and $A = \DS/I = \Apolar{f}$. By $\Dhdvect{\bullet}$ we denote the
        decomposition of the Hilbert function of $A$ and we set $e(a) := \sum_{t=0}^a
        \Dhd{t}{1}$.

        We say that $f$ \emph{is in the standard form} if
        \[
        f = f_0 + f_1 + f_2 + f_3 + \dots + f_s,\quad \mbox{where}\quad f_i\in \DP_i \cap
        k[x_1,\dots,x_{e(s-i)}]\mbox{ for all } i .
        \]
        Note that if $f$ is in the standard form and $\partial\in \DmmS$ then
        $f + \partial\hook f$ is also in the standard form.
        We say that an Artin Gorenstein algebra $\DS/I$ is in the \emph{standard
        form} if any (or every) dual socle generator of $\DS/I$ is in the
        standard form, see Proposition \ref{ref:standardformconds:prop} below.
    \end{defn}

    \begin{example}\label{ref:standardformofpowersum:example}
        If $f = x_1^6 + x_2^5 + x_3^3$, then $f$ is in the standard form.
        Indeed, $e(0) = 1$, $e(1) = 2$, $e(2) = 2$, $e(3) =
        3$ so that we should check that $x_1^6\in k[x_{1}]$, $x_2^5\in k[x_1,
        x_2]$, $x_3^3\in k[x_1, x_2, x_3]$, which is true. On the contrary, $g
        = x_3^6 + x_2^5 + x_1^3$ is not in the standard form, but may be put
        in the standard form via a change of variables.
    \end{example}

    The change of variables procedure of
    Example~\ref{ref:standardformofpowersum:example} may be generalised to
    prove that every
    local Artin Gorenstein algebra can be put in a standard form, as the following
    Proposition~\ref{ref:existsstandardform:prop} explains.

    \begin{prop}\label{ref:existsstandardform:prop}
        For every Artin Gorenstein algebra $\DS/I$ there is an automorphism
        $\varphi:\DS\to \DS$ such that $\DS/\varphi(I)$ is in the standard form.
    \end{prop}
    \begin{proof}
        See \cite[Thm 5.3AB]{ia94}, the proof is rewritten in \cite[Thm 4.38]{JelMSc}.
    \end{proof}

    The idea of the proof of Proposition~\ref{ref:existsstandardform:prop} is
    to ``linearise'' some elements of $\DS$. This is quite technical and
    perhaps it can be best seen on the following example.
    \begin{example}\label{ref:iarrfavorite:ex}
        On this example we exhibit the proof of
        Proposition \ref{ref:existsstandardform:prop}. Let $f = x_1^6 +
        x_1^4x_2$. The annihilator of $f$ in $\DS$ is $(\Dx_2^2, \Dx_1^5 -
        \Dx_1^3\Dx_2)$, the Hilbert function of $\Apolar{f}$ is $(1, 2, 2, 2,
        1, 1, 1)$ and the symmetric decomposition is
        \[\Dhdvect{0} = (1, 1, 1,
            1, 1, 1, 1),\ \ \Dhdvect{1} = (0, 0, 0, 0, 0, 0),\ \ \Dhdvect{2} = (0, 1, 1,
        1, 0).\]
        This shows that $e(0) = 1$, $e(1) = 1$, $e(2) = 2$. If $f$ is in the
        standard form we should have
        $f_5 = x_1^4x_2 \in k[x_1, \ldots, x_{e(1)}] = k[x_1]$. This means that $f$ is not
        in the standard form. The ``reason'' for $e(1) = 1$ is the fact that
        $\Dx_1^{3}(\Dx_2 - \Dx_1^2)$ annihilates $f$, and the ``reason'' for
        $f_5\not\in k[x_1]$ is that $\Dx_2 - \Dx_1^2$ is
        not a linear form. Thus we make $\Dx_2 - \Dx_1^2$ a linear form by twisting by a
        suitable automorphism of $\DS$.

        We define an automorphism $\psi:\DS\to \DS$ by
        $\psi(\Dx_1) = \Dx_1$ and $\psi(\Dx_2) = \Dx_2 + \Dx_1^2$, so that we
        have
        $\psi(\Dx_2 - \Dx_1^2) = \Dx_2$.
        The automorphism maps the annihilator of $f$ to the ideal
        ${I := ((\Dx_2 + \Dx_1^2)^2, \Dx_1^3\Dx_2)}$. We will see that the
        algebra $\DS/I$ is in the standard form and also find a particular dual generator obtained
        from $f$.

        As mentioned in Remark \ref{ref:dualautomorphisms:remark}, the automorphism $\psi$ induces an automorphism
        $\DPut{psistar}{\psi^*}$ of the $k$-linear space $\DP_{\leq
        6}$. This automorphism maps $f$ to a dual socle
        generator $\Dpsistar{f}$ of $\DS/I$.

        The element
        $F := \Dpsistar{x_1^6}$ is the only element of $\DP$ such that
        $\psi(\Dx_1^7)\hook F = \psi(\Dx_2)\hook F = 0$,
        $\psi(\Dx_1^6)(F) = 1$ and
        $\psi(\Dx_1^{l})(F) = 0$ for $l\leq 5$. Caution: in the last line we
        use evaluation on the functional and not the induced action (see Remark
        \ref{ref:dualautomorphisms:remark}).
        One can compute that $\Dpsistar{x_1^6} = x_1^6 - x_1^4x_2 + x_1^2x_2^2
        - x_2^3$ and similarly $\Dpsistar{x_1^4x_2} = x_1^4x_2 - 2x_1^2x_2 +
        3x_2^3$ so that $\Dpsistar{f} = x_1^6 - x_1^2x_2^2 + 2x_2^3$. Now
        indeed $x_1^6\in k[x_1], x_1^2x_2^2\in k[x_1, x_2]$ and $2x_2^3\in
        k[x_1, x_2]$ so the dual socle generator is in the standard form.
    \end{example}

    We note the following equivalent conditions for a dual socle generator to
    be in the standard form.
    \begin{prop}\label{ref:standardformconds:prop}
        In the notation of Definition \ref{ref:standardform:def}, the following
        conditions are equivalent for a polynomial $f\in \DP$:
        \begin{enumerate}
            \item the polynomial $f$ is in the standard form,
            \item for all $r$ and $i$ such that $r > e(s-i)$ we have
                $\DmmS^{i - 1}\Dx_r \subseteq I = (f)^{\perp}$. Equivalently,
                for all $r$ and $i$ such that $r > e(i)$ we have
                $\DmmS^{s - i - 1}\Dx_r \subseteq I = (f)^{\perp}$.
        \end{enumerate}
    \end{prop}
    \begin{proof}
        Straightforward.
    \end{proof}

    \begin{cor}\label{ref:automorphismpresever:cor}
        Let $f\in \DP$ be such that the algebra $\DS/I$ is in the standard
        form, where $I =\Dan{f}$.
        Let $\varphi$ be an automorphism of $\DS$ given by
        \[
            \varphi\pp{\Dx_i} = \kappa_i\Dx_i + q_i\mbox{ where }q_i\mbox{ is such
            that } \deg(q_i \hook f) \leq \deg(\Dx_i \hook f)\mbox{ and
            }\kappa_i\in k\setminus\left\{ 0 \right\}.
        \]
        Then the algebra $\DS/\varphi^{-1}(I)$ is also in the standard form.
    \end{cor}

    \begin{proof}
        The algebras $\DS/I$ and $\DS/\varphi^{-1}(I)$ are isomorphic, in
        particular they have equal functions $e(\cdot)$. By
        Proposition \ref{ref:standardformconds:prop} it suffices to
        prove that if for some $r, i$ we have $\DmmS^r \Dx_i \subseteq I$,
        then $\DmmS^r \Dx_i \subseteq \varphi^{-1}(I)$. The latter condition is equivalent to
        $\DmmS^r \varphi(\Dx_i) \subseteq I$.
        If $\DmmS^r \Dx_i \hook f = 0$ then $\deg (\Dx_i \hook f) < r$
        so, by assumption, $\deg (q_i \hook f) < r$ thus $\DmmS^r q_i \hook f
        = 0$ and
        $\DmmS^r \varphi(\Dx_i) = \DmmS^r (\kappa_i\Dx_i + q_i) \hook f = 0$.
    \end{proof}

    \begin{cor}\label{ref:basicautos:cor}
        Suppose that $q\in\DmmS^2$ does not contain $\Dx_i$ and let $\varphi:
        \DS\to \DS$ be an automorphism given by
        \[\varphi(\Dx_j) = \Dx_j\mbox{ for all }j\neq i\mbox{ and
        }\varphi(\Dx_i) = \kappa_i\Dx_i + q,\mbox{ where }\kappa_i\in
    k\setminus\{0\}.\]
        Suppose that $\DS/I$ is in the standard form, where $I = \Dan{f}$ and
        that $\deg(q\hook f) \leq \deg(\Dx_i \hook f)$. Then the algebras
        $\DS/\varphi(I)$ and $\DS/\varphi^{-1}(I)$ are also in the standard
        form.
    \end{cor}

    \begin{proof}
        Note that $\psi:\DS\to \DS$ given by $\psi(\Dx_j) = \Dx_j$ for $j\neq
        i$ and $\psi(\Dx_i) = \kappa_i^{-1}(\Dx_i - q)$ is an automorphism of $\DS$
        and furthermore $\psi(\kappa_i\Dx_i + q) = \Dx_i - q + q = \Dx_i$ so
        that $\psi = \varphi^{-1}$.
        Both $\varphi$ and $\psi$ satisfy assumptions of Corollary
        \ref{ref:automorphismpresever:cor} so both $\DS/\varphi^{-1}(I)$ and
        $\DS/\psi^{-1}(I) = \DS/\varphi(I)$ are in the standard form.
    \end{proof}

    \begin{remark}\label{ref:techbasicautos:remark}
        The assumption $q\in\DmmS^2$ of Corollary~\ref{ref:basicautos:cor} is
        needed only to ensure that $\varphi$ is an automorphism of $\DS$. On
        the other hand the fact that $q$ does not contain $\Dx_i$ is
        important, because it allows us to control $\varphi^{-1}$ and in
        particular prove that $S/\varphi(I)$ is in the standard form.
    \end{remark}

    The following Corollary \ref{ref:semibasicautos:ref} is a straightforward generalisation of Corollary
    \ref{ref:basicautos:cor}, but the notation is difficult. We first choose
    a set $\mathcal{K}$ of variables. The automorphism sends each variable
    from $\mathcal{K}$ to
    (a multiple of) itself plus a suitable polynomial in variables not
    appearing in $\mathcal{K}$.

    \begin{cor}\label{ref:semibasicautos:ref}
        Take $\mathcal{K} \subseteq \{1, 2, \dots, n\}$ and $q_i\in\DmmS^2$
        for $i\in \mathcal{K}$
        which do not contain any variables from the set $\{\Dx_i\}_{i\in
            \mathcal{K}}$. Define $\varphi:
        \DS\to \DS$ by
        \[
            \varphi(\Dx_i) = \begin{cases}
                \Dx_i &\mbox{ if }i\notin \mathcal{K}\\
                \kappa_i\Dx_i + q_i,\mbox{ where }\kappa_i\in
                k\setminus\{0\} &\mbox{ if } i\in \mathcal{K}.
            \end{cases}
        \]
        Suppose that $\DS/I$ is in the standard form, where $I = \Dan{f}$ and
        that $\deg(q_i\hook f) \leq \deg(\Dx_i \hook f)$ for all $i\in
        \mathcal{K}$. Then the algebras
        $\DS/\varphi(I)$ and $\DS/\varphi^{-1}(I)$ are also in the standard
        form.\qed
    \end{cor}

\section{Special forms of dual socle generators}\label{sec:specialforms}

Recall that $k$ is an algebraically closed field of characteristic
neither $2$ nor $3$.

In the previous section we mentioned that for every local Artin Gorenstein
algebra there exists a dual socle generator in the standard form, see Definition
\ref{ref:standardform:def}. In this section we will see that in most cases we
can say more about this generator. Our main aim is to put the generator in the
form $x^{s} + f$, where $f$ contain no monomial divisible by a ``high'' power of
$x$. We will use it to prove that families arising from certain ray
decompositions (see Definition \ref{ref:raydecomposition:def}) are flat.

We begin with an easy observation.
    \begin{remark}\label{ref:removinglinearpart:remark}
        Suppose that a polynomial $f\in \DP$ is such that
        $H_{\Apolar{f}}(1)$ equals the number of variables in $\DP$. Then any linear form
        in $\DP$ is a derivative of $f$. If $\deg f > 1$ then the
        $\DS$-submodules $\DS f$ and $\DS(f - f_1
        - f_0)$ are equal, so analysing this modules we may assume $f_1 =
        f_0 = 0$, i.e.~the linear part of $f$ is zero.

        Later we use this remark implicitly.
    \end{remark}

    The following Lemma~\ref{ref:topdegreetwist:lem} provides a method to
    slightly improve the given dual socle generator. This
    improvement is the building block of all other results in this section.

\begin{lem}\label{ref:topdegreetwist:lem}
    Let $f\in \DP$ be a polynomial of degree $s$ and $A$ be the apolar algebra of
    $f$. Suppose that $\Dx_1^s\hook f\neq 0$. For every $i$ let $d_i :=
    \deg(\Dx_1\Dx_i
    \hook f) + 2$.

    Then $A$ is isomorphic to the apolar algebra of a polynomial $\hat{f}$ of
    degree $s$, such that
    $\Dx_1^s \hook \hat{f} = 1$ and $\Dx_1^{d_i - 1}\Dx_i \hook \hat{f} = 0$ for all
    $i\neq 1$. Moreover, the leading forms of $f$ and $\hat{f}$ are equal up to a
    non-zero constant. If $f$ is in the standard form, then  $\hat{f}$ is also in the standard form.
\end{lem}

\begin{proof}
    By multiplying $f$ by a non-zero constant we may assume that $\DPut{xx}{\Dx_1}^s \hook f = 1$.
    Denote $I := \Dan{f}$.
    Since $\deg(\Dxx\Dx_i\hook f) = d_i - 2$, the polynomial $\Dxx^{d_i -
    1}\Dx_i\hook f = \Dxx^{d_i - 2}(\Dxx\Dx_i \hook f)$ is constant; we denote it by $\lambda_i$. Then
    \[\pp{\Dxx^{d_i - 1}\Dx_i - \lambda_i\Dxx^{s}}\hook f = 0,\mbox{ so that
    } \Dxx^{d_i - 1}\pp{\Dx_i - \lambda_i \Dxx^{s - d_i + 1}}\in I.\]
    Define an automorphism $\varphi:\DS\to \DS$ by
    \[
            \varphi(\Dx_i) = \begin{cases}
                \Dxx &\mbox{ if }i = 1\\
                \Dx_i - \lambda_i \Dxx^{s - d_i + 1} &\mbox{ if } i\neq 1,
            \end{cases}
    \]
    then $\alpha_1^{d_i - 1}\alpha_i \in \varphi^{-1}(I)$ for all $i > 1$. The
    dual socle generator $\hat{f}$ of the algebra $\DS/\varphi^{-1}(I)$ has the
    required form. We can easily check that the graded algebras of
    $\DS/\varphi^{-1}(I)$ and $\DS/I$ are equal, in particular $\hat{f}$ and $f$
    have the same leading form, up to a non-zero constant.

    Suppose now that $f$ is in the standard form.
    Let $i\in \{1,  \ldots ,n \}$. Then $d_i = \deg(\Dxx\Dx_i\hook f) + 2\leq \deg(\Dx_i
    \hook f) + 1$, so that $\deg(\Dxx^{s-d_i +1} \hook
    f) \leq d_i - 1\leq \deg(\Dx_i\hook f)$.
    Since $\varphi$ is
    an automorphism of $\DS$, by Remark~\ref{ref:techbasicautos:remark} we may
    apply Corollary~\ref{ref:semibasicautos:ref}
    to $\varphi$. Then $S/\varphi(I)$ is in
    the standard form, so $\hat{f}$ is in the standard form by definition.
\end{proof}

\begin{example}\label{ref:topdegreeexample:ex}
    Let $f\in k[x_1, x_2, x_3, x_4]$ be a polynomial of degree $s$. Suppose
    that the leading form $f_s$ of $f$ can be written as $f_s = x_1^s + g_s$
    where $g_s\in k[x_2, x_3, x_4]$. Then $\deg(\Dx_1\Dx_i\hook f) \leq s -
    3$ for all $i > 1$. Using Lemma~\ref{ref:topdegreetwist:lem} we produce $\hat{f} =
    x_1^s + h$ such that the apolar algebras of $f$ and $\hat{f}$ are
    isomorphic and
    $\Dx_1^{s-2}\Dx_i\hook h = 0$ for all $i\neq 1$. Then $\Dx_1^{s-2}\hook h
    = \lambda_1 x_1 + \lambda_2$, where $\lambda_i\in k$ for $i=1, 2$.
    After adding a suitable derivative to $\hat{f}$, we may assume $\lambda_1
    = \lambda_2 = 0$, i.e.~$\Dx_1^{s-2}\hook h = 0$.
\end{example}

\begin{example}\label{ref:standardformofstretched:ex}
    Suppose that a local Artin Gorenstein algebra $A$ of socle degree $s$ has
    Hilbert function equal to $(1, H_1, H_2, \dots, H_c, 1, \dots, 1)$. The
    standard form of the dual socle generator of $A$ is
    \[f = \DPut{ys}{x_1}^s + \kappa_{s-1}\Dys^{s-1} + \dots +
    \kappa_{c+2}\Dys^{c+2} + g,\] where $\deg g\leq c+1$ and
    $\kappa_{\bullet}\in k$. By adding a suitable
    derivative we may furthermore make all $\kappa_{i} = 0$ and assume
    that $\Dx_1^{c+1}\hook g = 0$. Using Lemma \ref{ref:topdegreetwist:lem} we
    may also assume that $\Dx_{1}^{c}\Dx_j\hook g = 0$ for every $j\neq 1$ so
    we may assume $\Dx_{1}^c \hook g = 0$, arguing as in
    Example~\ref{ref:topdegreeexample:ex}. This gives a dual socle generator
    \[f = x_1^s + g,\]
    where $\deg g \leq c+1$ and $g$ does not contain monomials divisible by
    $\Dys^c$.
\end{example}

The following proposition was proved in  \cite{CN2stretched}
under the assumption that $k$ is algebraically closed of characteristic
zero and in \cite[Thm
5.1]{JelMSc} under the assumption that $k = \mathbb{C}$. For completeness we include the
proof (with no further assumptions on $k$ other than the ones listed at the beginning of this
section).
\begin{prop}\label{ref:squares:prop}
    Let $A$ be Artin local Gorenstein algebra of socle degree $s\geq 2$ such that the Hilbert function decomposition from
    Theorem~\ref{ref:Hfdecomposition:thm} has $\Dhdvect{A, s-2} =
    (0, q, 0)$. Then $A$ is isomorphic to the apolar
    algebra of a polynomial $f$ such that $f$ is in the standard form and the
    quadric part $f_2$ of $f$ is a sum of
    $q$ squares of variables not appearing in $f_{\geq 3}$ and a quadric
    in variables appearing in $f_{\geq 3}$.
\end{prop}

\begin{proof}
    Let us take a standard dual socle generator $f\in \DP := k[x_1,\dots,x_n]$
    of the algebra $A$. Now we will twist $f$ to obtain the required form of
    $f_2$. We may assume that $H_{\Apolar{f}}(1) = n$.

    If $s = 2$, then the theorem follows from the fact that the quadric $f$ may be
    diagonalised. Assume $s\geq 3$.
    Let
    $\DPut{parte}{e} := e(s-3) = \sum_{t=0}^{s-3} \Dhd{A, t}{1}$. We have
    $\DPut{tote}{n} = e(s-2) = f + q$,
    so that $f_{\geq 3}\in k[x_1,\dots,x_{\Dparte}]$ and $f_2\in
    k[x_1,\dots,x_{\Dtote}]$. Note that $f_{\geq 3}$ is also in the standard
    form, so that every linear form in $x_1, \ldots ,x_{\Dparte}$ is a
    derivative of $f_{\geq 3}$, see Remark~\ref{ref:removinglinearpart:remark}.

    First, we want to assure that $\Dx_{\Dtote}^2\hook f \neq 0$. If
    $\Dx_{\Dtote}\hook f\in
    k[x_1,\dots,x_{\Dparte}]$ then there exists an operator $\partial\in \DmmS^2$ such
    that $\pp{\Dx_{\Dtote} - \partial}\hook f = 0$. This contradicts the fact
    that $f$ was in the standard form (see the discussion in Example~\ref{ref:iarrfavorite:ex}).
    So we get that $\Dx_{\Dtote} \hook f$ contains some $x_r$ for $r >
    \Dparte$, i.e.~$f$
    contains a monomial $x_rx_{\Dtote}$. A change
    of variables involving only $x_r$ and $x_{\Dtote}$ preserves the standard form and gives
    $\Dx_{\Dtote}^2 \hook f \neq 0$.

    Applying Lemma \ref{ref:topdegreetwist:lem} to $x_{\Dtote}$ we see that $f$ may
    be taken to be in the form $\hat{f} + x_{\Dtote}^2$, where $\hat{f}$ does not contain
    $x_{\Dtote}$, i.e. $\hat{f}\in k[x_1, \ldots ,x_{n-1}]$. We repeat the argument for $\hat{f}$.
\end{proof}

\begin{example}
    If $A$ is an algebra of socle degree $3$, then $H_A = (1, n, e, 1)$ for
    some $n$, $e$. Moreover, $n\geq e$ and the symmetric decomposition of $H_A$ is $(1, e, e, 1)
    + (0, n-e, 0)$. By Proposition \ref{ref:squares:prop} we see that $A$ is
    isomorphic to the apolar algebra of
    \[
        f + \sum_{e<i\leq n} x_i^2,
    \]
    where $f\in k[x_1, \ldots ,x_{e}]$.
    This claim was first proved by Elias and Rossi, see~\cite[Thm
    4.1]{EliasRossiShortGorenstein}.
\end{example}

\subsection{Irreducibility for fixed Hilbert function in two variables.}

Below we analyse local Artin Gorenstein algebras with Hilbert function $(1, 2,
2, \ldots )$. Such algebras are (in some cases) classified up to isomorphism in
\cite{EliasVallaAlmostStretched}, but rather than such classification we need
to know the geometry of their parameter space, which is analysed (among other
such spaces) in \cite{iarrobino_punctual}.

\def\HilbSr{\mathcal{H}\hspace{-0.25ex}\mathit{ilb\/}_r\Spec \DS}%
We need the following Proposition~\ref{ref:irreducibility_of_dual_socle:prop},
which is part of folklore. We thank J.~Buczy\'nski for explaining the
proof.

Let $r\geq 1$ be a natural number. By
$\HilbSr$ we denote the Hilbert
scheme of length $r$ subschemes of the power series ring $\DS$. It is called the
\emph{punctual Hilbert scheme} because as a set, $\HilbSr$ is equal to the set
of length $r$ subschemes of $\mathbb{P}^n$ supported at a single fixed point.

We recall a classical construction.
Let $V$ be a constructible subset of
$\DP_{\leq s}$. Assume that the apolar algebra $\Apolar{f}$ has length $r$ for
every closed point $f\in V$. Then we may construct the incidence scheme $\{(f,
    \Apolar{f})\}\to V$ which is a finite flat family over $V$ and thus we
    obtain a morphism from $V$ to $\HilbSr$. See \cite[Prop~4.39]{JelMSc} for
    details.

\begin{prop}\label{ref:irreducibility_of_dual_socle:prop}
    Let $\mathcal{R} \subseteq \HilbSr$ be a constructible
    subset and $V \subseteq \DP$ denote the set of all possible dual socle
    generators of elements of $\mathcal{R}$. If $\mathcal{R}$ is irreducible,
then also $V$ is irreducible.
\end{prop}

\begin{proof}
    \def\Of{O(f)}
    \def\Pf{\mathfrak{m}(f)}
    \def\Uf{U_f}
    \def\Hf{H_f}
    Below by $k^*$ and $\DS^*$ we denote the sets of invertible elements of $k$ and
    $\DS$ respectively.

    There is an induced surjective morphism $\varphi$ from $V$ to
        $\mathcal{R}$ as explained above. By construction the fiber over
    $\varphi(f)$ is $\DS^* \hook f$. The image $\mathcal{R}$ of
$\varphi$ is irreducible, so it is enough to show the existence of an open cover $\{H_i\}$
of $\mathcal{R}$ such that every $\varphi^{-1}(H_i)$ is irreducible.

    Choose an element $f\in V$ and a section of $\DmmS/\Dan{f}$ to $\DmmS$, that is, a
    linear subspace $\Pf \subseteq \DmmS$ such that
    $\Pf\to \DmmS/\Dan{f}$ is bijective. Let $\Of := \Pf
    + k \subseteq \DS$, then $\DS\hook f = \Of\hook f$. Finally let $\Of^* := k^* + \Pf$,
    so that $\varphi^{-1}(\varphi(f)) = \Of^*\hook f$.
    Consider the set
    \[\Uf = \left\{ g\in V\ |\ \Of \cap \Dan{g} = 0 \right\} =
        \left\{ g\in V\ |\ \Of\hook g = \DS\hook g\right\}.\]
    It is an open
    set in $V$ and its image $\Hf = \varphi(\Uf)$ is open (hence irreducible) in the Hilbert
    scheme. Moreover $\Uf = \varphi^{-1}(\Hf)$. For every $g\in \Uf$ the fiber
    $\varphi^{-1}(\varphi(g))$ is equal to $\Of^*\hook g$.

    By \cite[Proposition~18 and its Corollary]{emsalem} there is an open
    neighborhood $\Hf' \subseteq \Hf$ of $\varphi(f)$ such that the morphism $\varphi:
    \varphi^{-1}(\Hf')\to \Hf'$ has a section $i$. Denoting
    $\varphi^{-1}(\Hf')$ by $\Uf'$, we have a surjective morphism $\Of^*
    \times \Hf'  \to  \Uf'$ mapping $(\sigma, h)$ to $\sigma\hook i(h)$. Since
    $\Of^*$ and $\Hf'$ are irreducible, also $\Uf'$ is irreducible. Therefore
    $\{\Hf'\}$ form a desired cover of $\mathcal{R}$ and so $V$ is
irreducible.
\end{proof}

\begin{prop}\label{ref:irreducibleintwovariables:prop}
    Let $H = (1, 2, 2, *, \dots, *, 1)$ be a vector of length
    $s+1$. The set of
    polynomials $f\in k[x_1, x_2]$ such that $H_{\Apolar{f}} = H$ constitutes
    an irreducible subscheme of the affine space $k[x_1, x_2]_{\leq s}$.
    A general member of
    this set has, up to an automorphism of $\DP$ induced by an automorphism of $\DS$,
    the form $f + \partial\hook f$, where $f = x_1^s + x_2^{s_2}$ for some
    $s_2 \leq s$.
\end{prop}

\begin{proof}
    Let $V \subseteq k[x_1, x_2]$ denote the set of $f$ such that
    $H_{\Apolar{f}} = H$. Then the image of $V$ under the mapping sending $f$
    to $\Apolar{f}$ is irreducible by~\cite[Thm~3.13]{iarrobino_punctual}. By
    Proposition~\ref{ref:irreducibility_of_dual_socle:prop} the set $V$ is
irreducible.

    In the case $H = (1, 1, 1,  \ldots , 1)$ the claim (with $s_2 = 0$) follows directly from
    the existence of the standard form of a polynomial. Further in the proof we assume $H(1) = 2$.

    Let us take a general polynomial $f$ such that
    $H_{\Apolar{f}} = H$. Then $\Dan{f} = (q_1, q_2)$ is a complete
    intersection, where $q_1\in\DS$ has order $2$, i.e.~$q_1\in
    \DmmS^2\setminus \DmmS^3$. Since
    $f$ is general, we may assume that the quadric part of $q_1$ has maximal
    rank, i.e. rank two, see also
    \cite[Thm~3.14]{iarrobino_punctual}. Then after a change of variables $q_1
    \equiv \Dx_1\Dx_2 \mod \DmmS^3$.
    Since the leading form $\Dx_1\Dx_2$ of $q_1$ is reducible, $q_1 = \delta_1
    \delta_2$ for some $\delta_1, \delta_2\in \DS$ such that $\delta_i \equiv
    \Dx_i \mod \DmmS^2$ for $i=1,2$,
    see~e.g.~\cite[Thm~16.6]{Kunz_plane_algebraic_curves}. After an
    automorphism of $\DS$ we may assume $\delta_i = \Dx_i$, then $\Dx_1\Dx_2 =
    q_1$ annihilates $f$, so that it has the required form.
\end{proof}

\subsection{Homogeneous forms and secant
varieties}\label{sec:homogeneousforms}

\def\kchar{\operatorname{char}\, k}
It is well-known that if $F\in \DP_s$ is a form such that $H_{\Apolar{F}} = (1, 2,
\dots, 2, 1)$ then the standard form of $F$ is either $x_1^{s} +
x_2^s$ or $x_1^{s-1}x_2$. In particular the set of such forms in $\DP$ is
irreducible
and in fact it is open in the so-called secant variety. This section is
devoted to some generalisations of this result for the purposes of
classification of leading forms of polynomials in $\DP$.

The following proposition is well-known if the base field is of characteristic
zero (see \cite[Thm 4]{BGIComputingSymmetricRank} or \cite{LO}), but we could
not find a reference for the positive characteristic case, so for completeness we
include the proof.

\begin{prop}\label{ref:thirdsecant:prop}
    Suppose that $\DPut{ff}{F}\in k[x_1, x_2, x_3]$ is a homogeneous polynomial
    of degree $\Ddegf\geq 4$.
    The following conditions are equivalent
    \begin{enumerate}
        \item the algebra $\Apolar{\Dff}$ has Hilbert function $H$ beginning
            with $H(1) = H(2) = H(3) = 3$, i.e.~$H = (1, 3, 3, 3,
             \ldots )$,
        \item after a linear change of variables $\Dff$ is in one of the forms
            \[x_1^{\Ddegf} + x_2^{\Ddegf} + x_3^{\Ddegf},\qquad x_1^{\Ddegf -
            1}x_2 + x_3^{\Ddegf},\qquad x_1^{\Ddegf - 2}(x_1 x_3 + x_2^2).\]
    \end{enumerate}
    Furthermore, the set of forms in $k[x_1, x_2, x_3]_{\Ddegf}$ satisfying
    the above conditions is irreducible.
\end{prop}

\begin{proof}
    For the characteristic zero case see \cite{LO} and references therein.
    \def\span#1{\langle #1 \rangle}%

    \def\Dtmpy{\theta}
    Let $\DS = k[\Dx_1, \Dx_2,
    \Dx_3]$ be a polynomial ring dual to $\DP$. This notation is incoherent
    with the global notation, but it is more readable than $\DSpoly$.

    Let $I := \annn{\DS}{\Dff}$ and $I_2 := \span{\Dtmpy_1, \Dtmpy_2,
    \Dtmpy_3}
    \subseteq \DS_2$ be the linear space of operators of degree $2$
    annihilating $\Dff$. Let  $A := \DS/I$, $J := (I_2)
    \subseteq \DS$ and $B := \DS/J$.
    Since $A$ has length greater than $3\cdot 3 > 2^3$, the ideal $J$ is not a complete
    intersection. Let us analyse the Hilbert function of $A$. By symmetry of
    $H_A$, we have $H_A(s-1) = H_A(1) = 3$. By
    Remark~\ref{ref:MacaulaysBoundedGrowth:rmk} we have $3 = H_A(3) \geq H_A(4)
    \geq  \ldots \geq H_A(s-1) = 3$, thus
    \[
        H_A(m) = 3\quad\mbox{for all}\quad m = 1, 2,  \ldots , s-1.
    \]
    We will prove that the graded ideal $J$ is saturated and defines a zero-dimensional scheme of degree
    $3$ in $\mathbb{P}^2 = \Proj \DS$. First, $3 = H_A(3)
    \leq H_{B}(3) \leq 4$ by Macaulay's Growth Theorem. If
    $H_B(3) = 4$ then by Lemma~\ref{ref:P1gotzmann:lem} and
    Remark~\ref{ref:GorensteinSaturated:rmk} we have
    $H_A(1) = 2$, a contradiction. We have proved that $H_{B}(3) = 3$.

    Now we want to prove that $H_B(4) = 3$. By Macaulay's Growth Theorem applied to
    $H_B(3) = 3$ we have $H_B(4) \leq 3$. If $\Ddegf > 4$ then $H_A(4) = 3$,
    so $H_B(4) \geq 3$. Suppose $\Ddegf = 4$.  By Buchsbaum-Eisenbud result
    \cite{BuchsbaumEisenbudCodimThree} we know that the minimal number of
    generators of $I$ is odd. Moreover, we know that $A_n = B_n$ for $n < 4$,
    thus the generators of $I$ have degree two or four. Since $I_2$ is not a
    complete intersection, there are at least two generators of degree
    $4$, so $H_B(4) \geq H_A(4) + 2 = 3$.

    From $H_B(3) = H_B(4) = 3$ by Gotzmann's Persistence Theorem we see that $H_B(m) =
    3$ for all $m\geq 1$. Thus the scheme $\Gamma := V(J) \subseteq \Proj
    k[\Dx_1, \Dx_2, \Dx_3]$ is finite of degree $3$ and $J$ is saturated. In
    particular, the ideal $J = I(\Gamma)$ is contained in $I$.

    We will use $\Gamma$ to compute the possible forms of $F$, in the spirit
    of Apolarity Lemma, see \cite[Lem~1.15]{iakanev}. There are four possibilities for $\Gamma$:
    \begin{enumerate}
        \item $\Gamma$ is a union of three distinct, non-collinear points. After a change of basis $\Gamma
            = \left\{ [1:0:0] \right\} \cup \left\{ [0:1:0] \right\} \cup
            \left\{ [0:0:1] \right\}$, then $I_2 = (\Dx_1\Dx_2, \Dx_2\Dx_3,
            \Dx_3\Dx_1)$ and $\Dff = x_1^{\Ddegf} + x_2^{\Ddegf} + x_3^{\Ddegf}$.
        \item $\Gamma$ is a union of a point and scheme of length two, such
            that $\span{\Gamma} = \mathbb{P}^2$. After
            a change of basis $I_{\Gamma} = (\Dx_1^2, \Dx_1\Dx_2, \Dx_2\Dx_3)$,
            so that $\Dff = x_3^{\Ddegf-1}x_1 + x_2^{\Ddegf}$.
        \item $\Gamma$ is irreducible with support $[1:0:0]$ and it is not a
            $2$-fat point. Then $\Gamma$ is Gorenstein and so $\Gamma$ may
            be taken as the curvilinear scheme defined by $(\Dx_3^2, \Dx_2\Dx_3,
            \Dx_1\Dx_3 - \Dx_2^2)$. Then, after a linear change of variables,
            $\Dff = x_1^{\Ddegf-1}x_3 +
            x_2^2x_1^{\Ddegf-2}$.
        \item $\Gamma$ is a $2$-fat point supported at $[1:0:0]$. Then
            $I_{\Gamma} = (\Dx_2^2, \Dx_2\Dx_3, \Dx_3^2)$, so $F =
            x_1^{\Ddegf-1}(\lambda_2 x_2 + \lambda_3 x_3)$ for some $\lambda_2,
            \lambda_3\in k$. But then there is a degree one operator in $\DS$
            annihilating $F$, a contradiction.
    \end{enumerate}
    The set of forms $\Dff$ which are sums of three powers of
    linear forms is irreducible. To see that the forms satisfying the
    assumptions of the Proposition constitute an irreducible
    subset of $\DP_{\Ddegf}$ we observe that every $\Gamma$ as above is smoothable by
    \cite{CEVV}. The flat family proving the smoothability of $\Gamma$ induces a family
    $\Dff_t \to \Dff$, such that $\Dff_{\lambda}$ is a sum of three powers of linear
    forms for $\lambda\neq 0$, see \cite[Corollaire in Section 2]{emsalem}. See also \cite{bubu2010} for a generalisation of
    this method.
\end{proof}

\begin{prop}\label{ref:fourthsecant:prop}
    Let $\Ddegf \geq 4$.
    Consider the set $\DPut{set}{\mc{S}}$ of all forms $\DPut{ff}{F}\in k[x_1,
    x_2, x_3, x_4]$ of degree $\Ddegf$
    such that the apolar algebra of $\Dff$ has Hilbert function $(1, 4, 4,
    4, \dots, 4, 1)$. This set is irreducible and its general member has the
    form $\ell_1^{\Ddegf}
    + \ell_2^{\Ddegf} + \ell_3^{\Ddegf} + \ell_4^{\Ddegf}$, where $\ell_1$,
    $\ell_2$, $\ell_3$, $\ell_4$
    are linearly independent linear forms.
\end{prop}

\begin{proof}
    First, the set $\DPut{setzero}{\mc{S}_0}$ of forms equal to $\ell_1^4
    + \ell_2^4 + \ell_3^4 + \ell_4^4$, where $\ell_1$, $\ell_2$, $\ell_3$, $\ell_4$
    are linearly independent linear forms, is irreducible and contained in
    $\Dset$. Then, it is enough
    to prove that $\Dset$ lies in the closure of $\Dsetzero$.

    We follow the proof of Proposition
    \ref{ref:thirdsecant:prop}, omitting some details which can be found there.
    Let $S = k[\Dx_1, \Dx_2, \Dx_3, \Dx_4]$, $I := \Dan{\Dff}$ and $J :=
    (I_2)$. Set $A = S/I$ and $B = S/J$. Then
    $H_B(2) = 4$ and $H_B(3)$ is either $4$ or $5$. If $H_B(3) = 5$, then by
    Lemma~\ref{ref:P1gotzmann:lem} we have $H_B(1) = 3$, a contradiction. Thus
    $H_B(3) = 4$.

    Now we would like to prove $H_B(4) = 4$. By
    Macaulay's Growth Theorem $H_B(4) \leq 5$. By
    Lemma~\ref{ref:P1gotzmann:lem} $H_B(4) \neq 5$, thus $H_B(4) \leq 4$. If
    $\Ddegf > 4$ then $H_B(4) \geq H_A(4) \geq 4$, so we concentrate on
    the case $\Ddegf = 4$.
    Let us write the minimal free resolution of $A$, which is symmetric by
    \cite[Cor 21.16]{EisView}:
    \[
        0\to S(-8) \to S(-4)^{\oplus a}\oplus S(-6)^{\oplus 6} \to
        S(-3)^{\oplus b}\oplus S(-4)^{\oplus c} \oplus S(-5)^{\oplus b}\to
        S(-2)^{\oplus 6} \oplus S(-4)^{\oplus a}\to S.
    \]
    Calculating $H_A(3) = 4$ from the resolution, we get $b = 8$. Calculating
    $H_A(4) = 1$ we obtain $6 - 2a + c = 0$. Since $1 + a = H_B(4) \leq 4$ we have $a\leq 3$,
    so $a = 3$, $c = 0$ and $H_B(4) = 4$.

    Now we calculate $H_B(5)$. If $\Ddegf > 5$ then $H_B(5) = 4$ as before.
    If $\Ddegf = 4$ then extracting syzygies of $I_2$ from the above resolution
    we see that $H_B(5) = 4
    + \gamma$, where $0\leq \gamma\leq 8$, thus
    $H_B(5) = 4$ and $\gamma = 0$.
    If $\Ddegf = 5$, then the resolution of $A$ is
    \[
        0\to S(-9)\to S(-4)^{\oplus 3}\oplus S(-7)^{\oplus 6}\to S(-3)^{\oplus
        8}\oplus S(-6)^{\oplus 8}\to S(-5)^{\oplus 3}\oplus S(-2)^{\oplus 6}
        \to S.
    \]
    So $H_B(5) = 56 - 20\cdot 6 + 8 = 4$. Thus, as in the previous case we see that $J$ is the
    saturated ideal of a scheme $\Gamma$ of degree $4$. Then $\Gamma$
    is smoothable by \cite{CEVV} and its smoothing induces a family $\Dff_t\to
    \Dff$, where $\Dff_{\lambda}\in \Dsetzero$ for $\lambda\neq 0$.
\end{proof}

The following Corollary~\ref{ref:equationforsecant:cor} is a consequence of
Proposition~\ref{ref:fourthsecant:prop}. This corollary is not used in the
proofs of the main results, but it is of certain interest of its own and shows
another connection with secant varieties. For simplicity and to refer to some
results from \cite{LO}, we assume that $k = \mathbb{C}$, but
the claim holds for all fields of characteristic either zero or large
enough.

\def\catal{\varphi_{a, \Ddegf-a}}
To formulate the claim we introduce catalecticant matrices.
Let $\varphi_{a, \Ddegf-a}: \DS_{a}\times \DP_{\Ddegf}\to\DP_{\Ddegf-a}$ be the contraction mapping
applied to homogeneous polynomials of degree $\Ddegf$. For $F\in \DP_\Ddegf$
we obtain $\catal(F): \DS_a \to \DP_{\Ddegf-a}$, whose matrix is called the \emph{$a$-catalecticant
matrix}. It is straightforward to see that $\rk \catal(F) =
H_{\Apolar{F}}(a)$.

\def\floor#1{\left\lfloor #1 \right\rfloor}%
\begin{cor}\label{ref:equationforsecant:cor}
    Let $\Ddegf \geq 4$ and $k = \mathbb{C}$.
    The fourth secant variety to $\Ddegf$-th Veronese reembedding of
    $\mathbb{P}^n$ is a subset $\sigma_4(v_\Ddegf(\mathbb{P}^n)) \subseteq \mathbb{P}(P_s)$ set-theoretically defined
    by the condition $\rk \catal \leq 4$, where $a = \floor{\Ddegf/2}$.
\end{cor}

\begin{proof}
    \def\Dh#1{H(#1)}
    Since $H_{\Apolar{F}}(a) \leq 4$ for $F$ which is a sum of four powers of
    linear forms, by semicontinuity every $F\in \sigma_4(v_\Ddegf(\mathbb{P}^n))$ satisfies the
    above condition.

    Let $F\in \DP_\Ddegf$ be a form satisfying $\rk \catal(F) \leq 4$. Let $A = \Apolar{F}$ and $H = H_A$ be the Hilbert
    function of $A$. We want to reduce to the case where $\Dh{n} = 4$ for all $0 < n < \Ddegf$.

    First we show that $\Dh{n} \geq 4$ for all $0 < n < \Ddegf$.
    If $\Dh{1}\leq 3$, then the claim follows from~\cite[Thm~3.2.1~(2)]{LO},
    so we assume $\Dh{1} \geq 4$.
    Suppose that for some $n$ satisfying  $4 \leq n < \Ddegf$ we
    have $\Dh{n} < 4$. Then by Remark~\ref{ref:MacaulaysBoundedGrowth:rmk} we
    have $\Dh{m} \leq \Dh{n}$ for
    all $m \geq n$, so that $\Dh{1} = \Dh{\Ddegf-1} < 4$, a contradiction. Thus $\Dh{n}\geq
    4$ for all $n\geq 4$. Moreover, $\Dh{3}\geq 4$ by Macaulay's Growth Theorem. Suppose now
    that $\Dh{2}
    < 4$. By Theorem~\ref{ref:MacaulayGrowth:thm} the
    only possible case is $\Dh{2} = 3$ and $\Dh{3} = 4$. But then $\Dh{1} =2 <
    4$ by Lemma~\ref{ref:P1gotzmann:lem}, a contradiction. Thus we have proved
    that
    \begin{equation}\label{eq:lowerbound}
        \Dh{n} \geq 4\quad \mbox{for all}\quad 0 < n < \Ddegf.
    \end{equation}

    We have $\Dh{a} = 4$. If $\Ddegf\geq 8$, then $a\geq 4$, so by
    Remark~\ref{ref:MacaulaysBoundedGrowth:rmk} we have $\Dh{n}
    \leq 4$ for all $n > a$. Then by the symmetry $\Dh{n} = \Dh{\Ddegf - n}$ we have $\Dh{n}
    \leq 4$ for all $n$. Together with $\Dh{n}\geq 4$ for $0 < n < \Ddegf$, we have
    $\Dh{n} = 4$ for $0 < n < \Ddegf$. Then $F\in \sigma_4(v_\Ddegf(\mathbb{P}^n))$ by
    Proposition~\ref{ref:fourthsecant:prop}.
    If $a = 3$ (i.e.~$\Ddegf = 6$ or $\Ddegf = 7$), then $\Dh{4} \leq 4$ by
    Lemma~\ref{ref:P1gotzmann:lem} and we finish the proof as in the case
    $\Ddegf
    \geq 8$.
    If $\Ddegf = 5$, then $a = 2$ and the Hilbert function of $A$ is $(1, n, 4, 4,
    n, 1)$. Again by Lemma~\ref{ref:P1gotzmann:lem}, we have $n\leq 4$, thus
    $n = 4$ by \eqref{eq:lowerbound} and
    Proposition~\ref{ref:fourthsecant:prop} applies.
    If $\Ddegf = 4$, then $H = (1, n, 4, n, 1)$. Suppose $n\geq 5$, then
    Lemma~\ref{ref:P1gotzmann:lem} gives $n\leq 3$, a contradiction. Thus
    $n = 4$ and Proposition~\ref{ref:fourthsecant:prop} applies also to this case.
\end{proof}

Note that for $s\geq 8$ the Corollary~\ref{ref:equationforsecant:cor} was also
proved, in the case $k = \mathbb{C}$, in \cite[Thm~1.1]{bubu2010}.

\section{Ray sums, ray families and their flatness}\label{sec:ray}

Recall that $k$ is an algebraically closed field of characteristic
neither $2$ nor $3$.
Since $k[[\Dx_i]]$ is a discrete valuation ring, all its ideals have the form
$\Dx_i^{\DPut{ord}{\nu}}k[[\Dx_i]]$ for some $\Dord \geq 0$. We use this property to construct certain
decompositions of the ideals in the power series ring $\DS$.

\def\Drord#1#2{\operatorname{rord}_{#1}\left( #2 \right)}%
\def\Dcomp#1{\mathfrak{p}_{#1}}%
\begin{defn}\label{ref:order:def}
    Let $I$ be an ideal of finite colength in the power series ring
    $\DPut{tmpseries}{k[[\Dx_1,\dots,\Dx_n]]}$
    and $\pi_i:\Dtmpseries \onto k[[\Dx_i]]$ be the projection defined by
    $\pi_i(\Dx_j) = 0$ for $j\neq i$ and $\pi_i(\Dx_i) = \Dx_i$.

    The $i$-th \emph{ray order} of $I$ is a non-negative
    integer $\Dord  = \Drord{i}{I}$ such that $\pi_i(I) = (\Dx_i^\Dord )$.
\end{defn}

By the discussion above, the ray order is
well-defined.
Below by $\Dcomp{i}$ we denote the kernel of $\pi_i$; this is the ideal
generated by all variables except for $\Dx_i$.

\begin{defn}\label{ref:raydecomposition:def}
    Let $I$ be an ideal of finite colength in the power series ring
    $\DS = \Dtmpseries$. A \emph{ray decomposition} of $I$ with respect to
    $\Dx_i$ consists of
    an ideal $J \subseteq S$, such that $J \subseteq I \cap \Dcomp{i}$, together with an
    element $q\in \Dcomp{i}$ and $\Dord\in \mathbb{Z}_{+}$ such that
    \[
        I = J + (\Dx_i^{\Dord} - q)\DS.
    \]
\end{defn}

Note that from Definition \ref{ref:order:def} it follows that for every
$I$ and $i$ a ray decomposition (with $J = I\cap \Dcomp{i}$) exists and that
$\Dord = \Drord{i}{I}$ for every ray decomposition.

\begin{defn}\label{ref:rayfamily:def}
    \def\DJpoly{J_{poly}}%
    Let $\DS = k[[\Dx_1, \ldots ,\Dx_n]]$ and $\DSpoly = k[\Dx_1, \ldots
    ,\Dx_n] \subseteq \DS$.
    Let $I = J + \left( \Dx_{i}^\Dord - q \right)\DS$ be a ray
    decomposition of a finite colength ideal $I \subseteq \DS$. Let $\DJpoly =
    J\cap \DSpoly$. The associated \emph{lower ray family} is
    \[
        k[t] \to \frac{\DSpoly[t]}{\DJpoly[t] + (\Dx_i^\Dord -
        t\cdot\Dx_i - q)\DSpoly[t]},
    \]
    and the associated \emph{upper ray family} is
    \[
        k[t] \to \frac{\DSpoly[t]}{\DJpoly[t] + (\Dx_i^\Dord -
            t\cdot\Dx_i^{\Dord-1} -
    q)\DSpoly[t]}.
    \]
    If the lower (upper) family is flat over $k[t]$ we will call it a \emph{lower (upper) ray degeneration}.
\end{defn}

Note that the lower and upper ray degenerations agree for $\Dord = 2$.

\begin{remark}
In all considered cases the quotient $\DSpoly/J_{poly}$ will be finite over
$k$, so that every ray family will be finite over $k[t]$.
Then every ray degeneration will give a morphism to the Hilbert scheme.
We leave this check to the reader.
\end{remark}

\begin{remark}\label{ref:fibers:rmk}
    In this remark for simplicity we assume that $i = 1$ in
    Definition~\ref{ref:rayfamily:def}. Below we
    write $\Dx$ instead of $\Dx_1$.
    Let us look at the  fibers of the upper ray family from this definition
    in a special case, when $\Dx\cdot q\in J$.
    The fiber over $t=0$ is isomorphic to $\DS/I$.
    Let us take $\lambda\neq 0$ and analyse the fiber at $t=\lambda$. This
    fiber is supported at $(0, 0, \dots, 0)$ and at $(0, \dots, 0, \lambda, 0, \dots,
    0)$, where $\lambda$ appears on the $i$-th position.
    In particular, this shows that the
    existence of an upper ray degeneration proves that the algebra $\DS/I$ is
    limit-reducible; this is true also for the lower ray degeneration.

    Now $\Dx^{\Dord+1} - \lambda\Dx^\Dord$ is in the ideal defining the fiber
    of the upper ray family over $t
    = \lambda$. Now one
    may compute that
    near $(0, \dots, 0)$ the ideal defining the fiber is
    $(\lambda\Dx^{\Dord-1} - q) + J$. Similarly
     near $(0, \dots, 0, \lambda, 0, \dots, 0)$ it is $(\Dx - \lambda) + (q) +
     J$. The argument is similar (though easier) to the proof of
     Proposition~\ref{ref:fibersofray:prop}.
\end{remark}

Most of the families constructed in \cite{CEVV} and \cite{cn09} are ray
degenerations.

\begin{defn}\label{ref:raysum:def}
    For a non-zero polynomial $f\in \DP$ and $d\geq 2$ the $d$-th \emph{ray sum of $f$ with respect to
    a derivation $\partial\in \DmmS$} is a polynomial $g\in \DP[x]$ given by
    \[g = f + \DPut{specvar}{x}^{d}\cdot\partial \hook f +
        \Dspecvar^{2d}\cdot\partial^2\hook f +
    \Dspecvar^{3d}\cdot\partial^{3}\hook f + \dots\ .\]
\end{defn}

The following proposition shows that a ray sum naturally induces a ray
decomposition, which can be computed explicitly.

\begin{prop}\label{ref:raysumideal:prop}
    Let $g$ be the $d$-th ray sum of $f$ with respect to $\partial\in \DmmS$
    such that $\partial\hook f\neq 0$. Let $\Dx$ be an
    element dual to $x$, so that $P[x]$ and $\DPut{T}{T}:= \DS
    [[\Dx]]$ are dual.
    The annihilator of $g$ in $\DT$ is
    given by the formula
    \begin{equation}\label{eq:anndecomposition}
        \annn{\DT}{g} = \annn{\DS}{f} + \pp{\sum_{i=1}^{d-1} k\Dx^i}
        \annn{\DS}{\partial\hook f} + (\Dx^{d} - \partial)\DT,
    \end{equation}
    where the sum denotes the sum of $k$-vector spaces. In particular,
    the ideal $\annn{\DT}{g} \subseteq \DT$ is generated by $\annn{\DS}{f}$,
    $\Dx\annn{\DS}{\partial\hook f}$ and $\Dx^d - \partial$.
    The formula \eqref{eq:anndecomposition} is a ray decomposition of
    $\annn{\DT}{g}$ with respect to $\Dx$ and with $J = \annn{\DS}{f}\DT +
     \Dx\annn{\DS}{\partial\hook f}\DT$ and $q = \partial$.
\end{prop}

\begin{proof}
    It is straightforward to see that the right hand side of Equation
    \eqref{eq:anndecomposition} lies in $\annn{\DT}{g}$.
    Let us
    take any $\partial'\in \annn{\DT}{g}$. Reducing the powers of $\Dx$ using
    $\Dx^{d} - \partial$ we can write
    \[
        \partial' = \sigma_0 + \sigma_1\Dx + \dots +
        \sigma_{k-1}\Dx^{d-1},
    \]
    where $\sigma_{\bullet}$ do not contain $\Dx$. The action of this derivation
    on $g$ gives
    \[
        0 = \sigma_0 \hook f + \Dspecvar \sigma_{d-1}\partial \hook f +
        \Dspecvar^{2} \sigma_{d-2}\partial\hook f + \dots +
        \Dspecvar^{d-1}\sigma_1\partial \hook f + \Dspecvar^d\left( \dots \right).
    \]
    We see that $\sigma_0\in \annn{\DS}{f}$ and $\sigma_i \in
    \annn{\DS}{\partial\hook f}$ for $i \geq 1$, so the equality
    is proved. It is also clear that $J \subseteq \DmmS \DT$ and $\annn{\DT}{g} = J + (\Dx^d -
    \partial)\DT$, so that indeed we obtain a ray decomposition.
\end{proof}

\begin{remark}\label{ref:Hilbfunccouting:rmk}
    It is not hard to compute the Hilbert function of the apolar algebra of
    a ray sum in some special cases. We mention one
    such case below.
    Let $f\in \DP$ be a polynomial satisfying $f_2 = f_1 = f_0 = 0$ and
    $\partial\in \DmmS^2$ be such that  $\partial \hook f = \ell$ is a linear
    form, so that $\partial^2\hook f = 0$. Let $A = \Apolar{f}$ and $B =
    \Apolar{f + x^2\ell}$. The only different values of $H_A$ and $H_B$
    are $H_B(m) = H_A(m) + 1$ for $m=1, 2$. The $f_2 = f_1 = f_0 = 0$
    assumption is needed to ensure that the degrees of $\partial\hook f$ and
    $\partial\hook (f + x^2\ell)$ are equal for all $\partial$ not
    annihilating $f$.
\end{remark}

\subsection{Flatness of ray families}

\begin{prop}\label{ref:raysumflatness:prop}
    \def\DTpoly{\DT_{poly}}%
    \def\DJpoly{J_{poly}}%
    Let $g$ be the $d$-th ray sum with respect to $f$ and $\partial$. Then the
    corresponding upper and lower ray families are flat. Recall, that these
    families are explicitly given as
    \begin{equation}\label{eq:upperrayfamily}
        k[t] \to \frac{\DTpoly[t]}{\DJpoly[t] + (\Dx^{d} - t\Dx^{d-1} -
    \partial)\DTpoly[t]}\quad
        \mbox{ (upper ray family),}
    \end{equation}
    \begin{equation}\label{eq:lowerrayfamily}
        k[t] \to \frac{\DTpoly[t]}{\DJpoly[t] + (\Dx^{d} -t\Dx - \partial)\DTpoly[t]}\quad\quad
        \mbox{ (lower ray family),}
    \end{equation}
    where $\DTpoly$ is the fixed polynomial subring of $\DT$.
\end{prop}

\begin{proof}
    \def\DTpoly{\DT_{poly}}%
    We start by proving the flatness of
        Family~\eqref{eq:lowerrayfamily}.

    We want to use Proposition~\ref{ref:flatelementary:prop}. To simplify
    notation let $J := J_{poly}$.
    Denote by $\DPut{tmpII}{\mathfrak{I}}$ the ideal defining the family and
    suppose that some $z\in \DtmpII$ lies in $(t-\lambda)$ for some $\lambda\in k$.
    Write $z$ as $i + i_2\DPut{spec}{\pp{\Dx^{d} - t\Dx -
\partial}}$, where $i\in J[t]$, $i_2\in \DTpoly[t]$, and note that by
Remark~\ref{ref:flatnessremovet:remark} we may assume $i\in J$, $i_2\in
\DTpoly$. Since $z\in (t-\lambda)$, we have that
$i + i_2\DPut{spec}{(\Dx^d - \lambda\Dx - \partial)} = 0$, so
\[
    i_2\Dspec = -i\in J.
\]
    By Proposition~\ref{ref:raysumideal:prop} the ideal $J$ is homogeneous with
respect to grading by $\Dx$. More precisely it is equal to $J_0 + J_1\Dx$,
where $J_0 = \annn{\DS}{f}\DT,\ J_1 = \annn{\DS}{\partial\hook f}\DT$ are generated by elements not containing $\Dx$, so that $J$
is generated by elements of $\alpha$-degree at most one. We now check the
assumptions of Lemma~\ref{ref:decompositionhomog:lem}. Note that $\partial J
\subseteq J_0$ by definition of $J$. If $r\in \DTpoly$ is
such that $r\Dx^d\in J$, then $r\in J_1$, so that $r(\lambda\Dx + \partial)\in
\Dx J_1 + J_0 \subseteq J$. Therefore the assumptions are
satisfied and the Lemma shows that
    $i_2\Dx^d\in J$. Then $i_2\Dx\in J$, thus
    $i_2(\Dx^d - t\Dx)\in J[t] \subseteq(\DtmpII \cap \DTpoly)[t]$. Since $i_2 \partial\in
    \DtmpII \cap \DTpoly$ by definition, this
    implies that $i + i_2(\Dx^d - t\Dx - \partial)\in J[t] \subseteq (\DtmpII\cap
    \DTpoly)[t]$. Now the flatness follows from
    Proposition~\ref{ref:flatelementary:prop}.

    The same proof works equally well for upper ray family: one should just
    replace $\Dx$  by $\Dx^{d-1}$ in appropriate places of the proof. For this reason we
    leave the case of Family~\eqref{eq:upperrayfamily} to the reader.
\end{proof}

\begin{prop}\label{ref:fibersofray:prop}
    Let us keep the notation of Proposition \ref{ref:raysumflatness:prop}.
    Let $\lambda\in k\setminus\left\{ 0 \right\}$.
    The fibers of the Family \eqref{eq:upperrayfamily} and Family
    \eqref{eq:lowerrayfamily} over $t-\lambda$ are
    reducible.

    Suppose that $\partial^2\hook f = 0$ and the characteristic of $k$ does
    not divide $d-1$. The fiber of the Family \eqref{eq:lowerrayfamily} over $t-\lambda$ is
    isomorphic to \[\Spec \Apolar{f} \sqcup \left(\Spec \Apolar{\partial
    f}\right)^{\sqcup d-1}.\]
\end{prop}

\begin{proof}
    For both families the support of the fiber over $t - \lambda$ contains the
    origin. The support of the fiber of Family \eqref{eq:upperrayfamily} contains
    furthermore a point with $\alpha = \lambda$ and other coordinates equal to
    zero. The support of the fiber of Family \eqref{eq:lowerrayfamily} contains a
    point with $\alpha = \omega$, where $\omega^{d-1} = \lambda$.

    Now let us concentrate on Family \eqref{eq:lowerrayfamily} and on the case
    $\partial^2\hook f = 0$.
    The support of the fiber over $t-\lambda$ is $(0,\dots,0,0)$ and
    $(0, \dots, 0, \omega)$, where $\omega^{d-1} = \lambda$ are $(d-1)$-th roots of
    $\lambda$, which are pairwise different because of the characteristic assumption.
    We will analyse the support point by point.
    By hypothesis $\partial\in \Dan{\partial\hook f}$, so that $\alpha\cdot \partial\in J$, thus $\alpha^{d+1} -
    \lambda\cdot \alpha^2$ is in the ideal $I$ of the fiber over $t = \lambda$.

    \def\DTpoly{\DT_{poly}}%
    Near $(0,0,\dots,0)$ the element $\alpha^{d-1} - \lambda$ is invertible, so
    $\alpha^2$ is in the localisation of the ideal $I$, thus $\alpha + \lambda^{-1}\partial$ is in the
    ideal. Now we check that the localisation of $I$ is equal to $\Dan{f} +
    (\alpha +
    \lambda^{-1}\partial)\DTpoly$. Explicitly, one should check that
    \[
        \pp{\Dan{f} + (\alpha + \lambda^{-1}\partial)\DTpoly}_{(0, \ldots ,0)}
        = \pp{\Dan{f} + (\Dx^{d}
    -\lambda\Dx - \partial)\DTpoly}_{(0,  \ldots ,0)}.
    \]
    Then the stalk of the fiber at $(0,
    \ldots , 0)$ is isomorphic to $\Spec \Apolar{f}$.

    Near $(0, 0,\dots, 0, \omega)$ the elements $\alpha$ and
    $\frac{\alpha^{k+1} - \lambda\cdot \alpha^2}{\alpha - \omega}$
    are invertible, so
    $\Dan{\partial\hook f}$ and $\alpha - \omega$ are in the localisation of
    $I$. This, along with
    the other inclusion, proves
    that this localisation  is generated by $\Dan{\partial\hook f}$ and $\alpha -
    \omega$ and thus the stalk of the fiber is isomorphic to $\Spec \Apolar{\partial f}$.
\end{proof}

We make the most important corollary explicit:
\begin{cor}\label{ref:smoothabilityofrayiff:cor}
    We keep the notation of Proposition \ref{ref:raysumflatness:prop}. Suppose
    that $\kchar$ does not divide $d-1$ and $\partial^2\hook f = 0$.
    If both apolar algebras of $f$ and $\partial\hook f$ are smoothable then
    also the apolar algebra of every ray sum of $f$ with respect to $\partial$
    is smoothable.\qed
\end{cor}

\begin{example}\label{ref:squareadding:example}
    Let $f\in k[x_1, \ldots ,x_n]$ be a dual socle generator of an algebra
    $A$. Then the algebra $B = \Apolar{f + x_{n+1}^{2}}$ is limit-reducible: it
    is a limit of algebras of the form $A \times k$. In particular, if $A$ is
    smoothable, then $B$ is also smoothable.

    Combining this with Proposition~\ref{ref:squares:prop}, we see that every
    local Gorenstein algebra $A$ of socle degree $s$ with $\Dhdvect{A, s-2} = (0, q,
    0)$, where $q\neq 0$, is limit-reducible.

    If $\deg f \geq 2$, then the Hilbert functions of $A = \Apolar{f}$ and $B = \Apolar{f + x_{n+1}^2}$
    are related by $H_{B}(m) = H_A(m)$ for $m\neq 1$ and $H_B(1) = H_A(1) +
    1$.
\end{example}

Above, we took advantage of the explicit form of ray decompositions coming
from ray sums to analyse the resulting ray families in depth. In
Proposition~\ref{ref:stretchedhavedegenerations:prop} below we prove
the flatness of the upper ray family without such knowledge. The price paid
for this is the fact that we get no information about the fibers of this
family.

\begin{prop}\label{ref:stretchedhavedegenerations:prop}
    Let $f = x_1^\Ddegf  + g\in \DP$ be a polynomial of degree $\Ddegf $ such that
    $\Dx_1^{c} \hook g = 0$ for some $c$ satisfying $2c\leq \Ddegf $. Then
    any ray decomposition $\Dan{f} = (\Dx_1^{\Dord} - q) + J$, where $J = \Dan{f}
    \cap (\Dx_2, \ldots ,\Dx_n)$,
    gives rise to an upper ray degeneration. In particular $\Apolar{f}$ is
    limit-reducible.
\end{prop}

\begin{proof}
    Let $\DPut{tmpII}{\mathfrak{I}} := \DPut{spec}{(\Dx_1^{\Dord} -
    t\Dx_{1}^{\Dord-1}
- q)} + J$ be the ideal defining the
    ray family and recall that $q, J \subseteq \Dcomp{1}$, where $\Dcomp{1} =
    (\Dx_2,\dots, \Dx_n)$.

    Since $\Dx_1^\Dord - q\in \Dan{f}$, we have $q\hook g = q\hook f =
    \Dx_1^{\Dord}\hook f = x_1^{\Ddegf -\Dord} + \Dx_1^\Dord \hook g$. Then
    $\Dx_1^{\Ddegf -\Dord}(q\hook
    g) = \Dx_1^{\Ddegf -\Dord}\hook x_1^{\Ddegf -\Dord} + \Dx_1^\Ddegf \hook g = 1$, thus
    $\Dx_1^{\Ddegf -\Dord}\hook g\neq 0$.
    It follows that $\Ddegf  - \Dord \leq c-1$, so $\Dord - 1\geq \Ddegf  - c\geq c$, thus
    $\Dx_{1}^{\Dord -1} \hook g = 0$. For all $\gamma\in \Dcomp{1}$, we claim that
    \begin{equation}\label{eq:contstretched}
        \gamma\cdot (\Dx_1^{\Dord } - t\Dx_1^{\Dord -1} - q)\in J[t].
    \end{equation}
    Note that $(\Dx_1^\Dord  - q) \hook f = 0$ and $\Dx_1^{\Dord -1}\gamma\hook f =
    \Dx_1^{\Dord -1}\gamma\hook g = 0$. This means that $\Dx_1^{\Dord -1}\gamma\in J$. Since always $(\Dx_1^{\Dord } - q)\gamma\in J$,
    we have proved \eqref{eq:contstretched}.

    Let $\DtmpII \subseteq \DS_{poly}[t]$ be the ideal
    defining the upper ray family. Take any $\lambda\in k$ and an element
    $i\in \DtmpII \cap (t-\lambda)$. We will prove that
    $i\in\DtmpII(t-\lambda) + \DPut{tmpzero}{\DtmpII_0}[t]$, where $\Dtmpzero = \DtmpII \cap
    \DS$, then Proposition~\ref{ref:flatelementary:prop} asserts that
    $\DS[t]/\DtmpII$ is flat. Write $i = i_1 + i_2\Dspec$. As before, we may
    assume $i_1\in J$, $i_2\in \DS$. Since $i\in (t-\lambda)$, we have
    $i_1 + i_2(\Dx_1^\Dord  - \lambda\Dx_1^{\Dord -1} - q) = 0$.
    Since $i_1\in \Dcomp{1}$, we also have $i_2\in \Dcomp{1}$. But then by
    Inclusion~\eqref{eq:contstretched} we have $i_2\Dspec \subseteq
    \Dtmpzero[t]$. Since clearly $i_1\in J \subseteq \Dtmpzero[t]$, the
    assumptions of Proposition~\ref{ref:flatelementary:prop} are satisfied,
    thus the upper ray family is flat.

    Now, Remark~\ref{ref:fibers:rmk} shows that a general fiber of the upper
    ray degeneration is reducible, thus $\Apolar{f}$ is a flat limit of
    reducible algebras, i.e.~limit-reducible.
\end{proof}

\begin{example}\label{ref:quarticlimitreducible:example}
    Let $f\in k[x_1, x_2, x_3, x_4]$ be a polynomial of degree $4$. Suppose
    that the leading form $f_4$ of $f$ can be written as $f_4 = x_1^4 + g_4$
    where $g_4\in k[x_2, x_3, x_4]$. We will prove that $\Apolar{f}$ is
    limit-reducible.
    By Example~\ref{ref:topdegreeexample:ex} we may
    assume that $f = x_1^4 + g$, where $\Dx_1^2\hook g = 0$. By
    Proposition~\ref{ref:stretchedhavedegenerations:prop} we see that
    $\Apolar{f}$ is limit-reducible.
\end{example}

\begin{example}\label{ref:stretched:example}
    Suppose that an Artin local Gorenstein algebra $A$ has Hilbert function
    $H_A = (1, H_1,\dots, H_c, 1,\dots, 1)$ and socle degree $\Ddegf  \geq 2c$.
    By Example \ref{ref:standardformofstretched:ex} we
    may assume that $A  \simeq \Apolar{x_1^{\Ddegf } + g}$, where $\Dx_1^{c} \hook g =
    0$ and $\deg g\leq c+1$. Then by Proposition \ref{ref:stretchedhavedegenerations:prop} we
    obtain a flat degeneration
    \begin{equation}\label{eq:exampledegeneration}
        k[t] \to \frac{\DS[t]}{(\Dx_1^{\Dord } - t\Dx_1^{\Dord -1} - q) + J}.
    \end{equation}
    Thus $A$ is limit-reducible in the sense of Definition
    \ref{ref:limitreducible:def}.
    Let us take $\lambda\neq 0$.
    By Remark \ref{ref:fibers:rmk} the fiber over $t = \lambda$ is
    supported at $(0, 0,\dots, 0)$ and at $(\lambda, 0,\dots, 0)$ and
    the ideal defining this fiber near $(0, 0,\dots, 0)$ is $I_0 = (\lambda\Dx_1^{\Dord -1} - q)
    + J$.
    From the proof of \ref{ref:stretchedhavedegenerations:prop} it follows that
    $\Dx_1^{\Dord -1} \hook g = 0$. Then one can check that $I_0$ lies in the
    annihilator of $\lambda^{-1} x_{1}^{\Ddegf -1} + g$.
    Since $\sigma\hook (x_1^{\Ddegf} + g) = \sigma \hook(\lambda^{-1} x_1^{\Ddegf - 1} +
    g)$ for every $\sigma\in (\Dx_2, \ldots ,\Dx_n)$, one calculates that the
    apolar algebra of $\lambda^{-1} x_1^{\Ddegf -1} + g$ has Hilbert function
    $(1, H_1,\dots, H_c, 1,\dots, 1)$ and socle degree $\Ddegf -1$. Then
    $\dimk \Apolar{x_1^{\Ddegf -1} + g} = \dimk \Apolar{\lambda^{-1}x_1^{\Ddegf } + g} -
    1$. Thus the fiber is a union of a point and
    $\Spec\Apolar{\lambda^{-1}x_1^{\Ddegf } + g}$,
    i.e.~degeneration~\eqref{eq:exampledegeneration} peels one point off $A$.
\end{example}

\subsection{Tangent preserving ray
degenerations}\label{subsec:tangentpreserving}

A (finite) ray degeneration gives a morphism from $\Spec k[t]$ to the Hilbert
scheme, i.e. a curve on the Hilbert scheme $\Hilb{n}{}$. In this section we prove
that in some cases
the dimension of the tangent space to $\Hilb{n}{}$ is constant along this curve.
This enables us to prove that certain points of this scheme are
smooth without the need for lengthy computations.

This section seems to be the
most technical part of the paper, so we include even more examples. The most
important results here are Theorem \ref{ref:nonobstructedconds:thm}
together with Corollary \ref{ref:CIarenonobstructed:cor}; see examples below
Corollary \ref{ref:CIarenonobstructed:cor} for applications.

Recall (e.g.~\cite[Prop~4.10]{JelMSc} or \cite{cn09}) that the dimension of the tangent space to
$\Hilb{n}{}$ at a $k$-point corresponding to a Gorenstein scheme $\Spec S/I$ is
$\dimk S/I^2 - \dimk S/I$.

\begin{lem}\label{ref:tangentflatcondition:lem}
    Let $d\geq 2$. Let $g$ be the $d$-th ray sum of $f\in \DP$ with respect to
    $\partial\in
    \DS$ such that $\partial^2 \hook f = 0$.
    Denote $I := \annn{\DS}{f}$ and $J := \annn{\DS}{\partial\hook f}$.
    Take $\DPut{T}{T} = S[[\Dx]]$ to be the ring dual to $\DP[x]$ and let
    \[\DPut{II}{\mathfrak{I}} := \pp{I + J\Dx +
        \DPut{spec}{(\Dx^{d} - t\Dx - \partial)}}\cdot \DT[t]\] be the ideal in $\DT[t]$ defining the
    associated lower ray degeneration, see Proposition \ref{ref:raysumflatness:prop}.
    Then the family $k[t] \to \DT[t]/\DII^2$ is flat if and only if $(I^2 :
    \partial) \cap I \cap J^2 \subseteq I\cdot J$.
\end{lem}

\begin{proof}

    To prove flatness we will use
    Proposition~\ref{ref:flatelementary:prop}.
    Take an element $i\in \DII^2\cap (t-\lambda)$. We want to prove that $i\in
    \DII^2 (t-\lambda) + \DPut{zerocomp}{\DII_0[t]}$, where $\Dzerocomp =
    \DII^2 \cap \DT$. Let $\DPut{JJ}{\mathcal{J}} := (I + J\Dx)\DT$.
    Subtracting a suitable element of $\DII^2(t-\lambda)$ we
    may assume that
    \[i = i_1 + i_2\Dspec + i_3\Dspec^2,\] where $i_1\in
    \DJJ^2$, $i_2\in \DJJ$ and $i_3\in \DT$.
    We will in fact show that $i\in \DII^2(t-\lambda) + \DJJ^2[t]$.

    To simplify notation denote $\DPut{ss}{\sigma} = \Dx^d - \lambda\Dx -
    \partial$. Note that $J\Dss\subseteq \DJJ$.
    We have $i_1 + i_2\Dss + i_3\Dss^2 = 0$. Let $j_3 := i_3\Dss$.
    We want to apply Lemma~\ref{ref:decompositionhomog:lem}, below we check
    its assumptions.
    The ideal $\DJJ$ is homogeneous with
    respect to $\Dx$, generated in degrees less than $d$. Let $s\in
    \DT$ be an element satisfying $s\Dx^{d}\in \DJJ$.
    Then $s\in J$, which implies $s(\lambda \Dx + \partial)\in \DJJ$.
    By Lemma~\ref{ref:decompositionhomog:lem} and $i_3\Dss^2 = j_3\Dss\in \DJJ$ we
    obtain
    $j_3\Dx^{d} \in \DJJ$, i.e.~$i_3 \Dss \Dx^d\in \DJJ$.
    Applying the same
    argument to $i_3\Dx^d$ we obtain $i_3\Dx^{2d}\in \DJJ$, therefore $i_3\in J\DT$.
    Then
    \[
        i_3 \Dspec^2 - i_3 \Dss \Dspec = i_3\Dx (t - \lambda) \Dspec \in
        \DJJ(t-\lambda) \Dspec \subseteq \DII^2(t-\lambda).
    \]
    Subtracting this element from $i$ and substituting $i_2 := i_2 + i_3 \Dss$
    we may assume $i_3 = 0$.
    We obtain
    \begin{equation}\label{eq:flatnesslemma}
        0 = i_1 + i_2\Dss = i_1 + i_2(\Dx^d - \lambda\Dx - \partial).
    \end{equation}
    Let $i_2 = j_2 + v_2\Dx$, where $j_2\in \DS$, i.e.~it does not contain
    $\Dx$. Since $i_2\in \DJJ$, we have $j_2\in I$. As before, we have $v_2\Dx (\Dspec - \Dss) = v_2\Dx^2(t-\lambda)\in
    \DII^2(t-\lambda)$, so that we may assume $v_2 = 0$.

    Comparing the top $\Dx$-degree terms of \eqref{eq:flatnesslemma} we see
    that $j_2\in J^2$.
    Comparing the terms of
    \eqref{eq:flatnesslemma} not containing $\Dx$, we deduce that
    $j_2\partial\in I^2$, thus $j_2\in (I^2:\partial)$. Jointly, $j_2\in I\cap
    J^2\cap (I^2:\partial)$, thus $j_2\in IJ$ by assumption.
    But then $j_2\Dx \in \DJJ^2$, thus $j_2\Dspec\in \DJJ^2[t]$ and since
    $i_1\in \DJJ^2$, the
    element $i$ lies in $\DJJ^2[t] \subseteq \Dzerocomp$. Thus the assumptions of
    Proposition~\ref{ref:flatelementary:prop} are satisfied and the family
    $\DT[t]/\DII^2$ is flat over $k[t]$.

    The converse is easier: one takes $i_2\in I\cap
    J^2\cap (I^2:\partial)$ such that $i_2\not\in IJ$. On one hand, the element $j:=i_2(\Dx^d -
    \partial)$ lies in $\DJJ^2$ and we get that $i_2\Dspec - j = ti_2\Dx\in
    \DII^2$. On the other hand if $i_2\Dx\in \DII^2$, then $i_2\Dx\in (\DII^2 + (t))
    \cap \DT = (\DJJ + (\Dx^d - \partial))^2$, which is not the
    case.
\end{proof}

\begin{remark}\label{ref:tangentfibers:rmk}
    \def\tansp#1{\tan(#1)}%
    Let us keep the notation of Lemma~\ref{ref:tangentflatcondition:lem}. Fix
    $\lambda\in k\setminus \{0\}$ and suppose that the characteristic of
    $k$ does not divide $d-1$.
    The supports of the fibers of $\DS[t]/\DII$, $\DII/\DII^2$ and
    $\DS[t]/\DII^2$ over $t = \lambda$ are finite and equal.
    In particular from Proposition \ref{ref:fibersofray:prop} it follows that
    the dimension of the fiber of $\DII/\DII^2$ over $t-\lambda$ is equal to
    $\tansp{f} + (d-1)\tansp{\partial\hook f}$, where $\tansp{h} = \dimk \Dan{h}/\Dan{h}^2$ is the dimension of
    the tangent space to the point of the Hilbert scheme corresponding to
    $\Spec \DS/\Dan{h}$.
\end{remark}

\begin{thm}\label{ref:nonobstructedconds:thm}
    Suppose that a polynomial $f\in \DP$ corresponds to a smoothable, unobstructed  algebra
    $\Apolar{f}$. Let $\partial\in \DS$ be such that $\partial^2\hook f = 0$
    and the algebra $\Apolar{\partial\hook f}$ is smoothable and unobstructed.
    The following are equivalent:
    \begin{enumerate}
        \item[1.]\label{it:somenonob} the $d$-th ray sum of $f$ with respect to
            $\partial$ is unobstructed for some $d$ such that $2\leq d \leq
            \kchar$ (or $2\leq d$ if $\kchar = 0$).
        \item[1a.]\label{it:allnonob} the $d$-th ray sum of $f$ with respect to
            $\partial$ is unobstructed for all $d$ such that $2\leq d \leq \kchar
            $ (or $2\leq d$ if $\kchar = 0$).
        \item[2.]\label{it:tgflat} The $k[t]$-module  $\DPut{defid}{\DII}/\Ddefid^2$ is flat,
            where $\Ddefid$ is the ideal defining the lower ray family of the
            $d$-th ray sum for some $2\leq d \leq \kchar$ (or $2\leq d$ if
            $\kchar = 0$), see
            Definition \ref{ref:rayfamily:def}.
        \item[2a.]\label{it:tgflatevery} The $k[t]$-module $\DPut{defid}{\DII}/\Ddefid^2$ is flat,
            where $\Ddefid$ is the ideal defining the lower ray family of the
            $d$-th ray sum for every $2\leq d \leq  \kchar$ (or $2\leq d$ if
            $\kchar = 0$), see
            Definition \ref{ref:rayfamily:def}.
        \item[3.]\label{it:quoflat} The family $k[t] \to \DS[t]/\Ddefid^2$ is flat,
            where $\Ddefid$ is the ideal defining the lower ray family of the
            $d$-th ray sum for some $2\leq d \leq \kchar$ (or $2\leq d$ if
            $\kchar = 0$).
        \item[3a.]\label{it:quoflatevery} The family $k[t] \to \DS[t]/\Ddefid^2$ is flat,
            where $\Ddefid$ is the ideal defining the lower ray family of the
            $d$-th ray sum for every $2\leq d \leq \kchar$ (or $2\leq d$ if
            $\kchar = 0$).
        \item[4.] The following inclusion (equivalent to equality) of ideals in
            $\DS$
            holds: $I\cap J^2 \cap (I^2:\partial) \subseteq I\cdot J$, where
            $I = \annn{\DS}{f}$ and $J = \annn{\DS}{\partial\hook f}$.
    \end{enumerate}
\end{thm}

\begin{proof}
    It is straightforward to check that the inclusion $I\cdot J\subseteq I\cap J^2 \cap
    (I^2:\partial) \subseteq I\cdot J$ in Point 4 always holds,
    thus the other inclusion is
    equivalent to equality.\\
    3. $\iff$ 4. $\iff$ 3a.
    The equivalence of Point 3 and
    Point 4 follows from Lemma
    \ref{ref:tangentflatcondition:lem}. Since Point 4 is independent of $d$,
    the equivalence of Point 4 and Point 3a also follows.

    2. $\iff$ 3. and 2a. $\iff$ 3a.
    We have an exact sequence of
    $k[t]$-modules
    \[0\to \Ddefid/\Ddefid^2 \to \DS[t]/\Ddefid^2 \to \DS[t]/\Ddefid
\to 0.\] Since $\DS[t]/\Ddefid$ is a flat
    $k[t]$-module by Proposition \ref{ref:raysumflatness:prop}, we see from
    the long exact sequence of $\operatorname{Tor}$ that
    $\Ddefid/\Ddefid^2$ is flat if and only if $\DS[t]/\Ddefid^2$ is flat.

    1. $\iff$ 2. and 1a. $\iff$ 2a.
    Let $g\in P[x]$ be the $d$-th ray sum of $f$ with
    respect to $\partial$. We may consider
    $\Apolar{g}$, $\Apolar{f}$, $\Apolar{\partial\hook f}$ as quotients of a
    polynomial ring $T_{poly}$, corresponding to points of the Hilbert scheme.
    The dimension of the tangent space at $\Apolar{g}$ is given by $\dimk
    \Ddefid/\Ddefid^2 \tensor k[t]/t = \dimk \Ddefid/(\Ddefid^2 + (t))$. By Remark \ref{ref:tangentfibers:rmk} it is
    equal to the sum of the dimension of the tangent space at $\Apolar{f}$ and
    $(d-1)$ times the dimension of the tangent space to $\Apolar{\partial\hook f}$. Since both
    algebras are smoothable and unobstructed we conclude that $\Apolar{g}$ is also
    unobstructed. On the other hand, if $\Apolar{g}$ is unobstructed, then
    $\Ddefid/\Ddefid^2$ is a finite $k[t]$-module such that the length of
    the fiber $\Ddefid/\Ddefid^2\tensor k[t]/\mathfrak{m}$ does not depend on the
    choice of the maximal ideal $\mathfrak{m} \subseteq k[t]$. Then
    $\Ddefid/\Ddefid^2$ is flat by \cite[Ex~II.5.8]{HarAG} or
    \cite[Thm~III.9.9]{HarAG} applied to the associated sheaf.
\end{proof}

\begin{remark}
    The condition from Point 4 of Theorem
    \ref{ref:nonobstructedconds:thm} seems very technical. It is
    enlightening to look at the images of $(I^2:\partial)\cap I$ and $I\cdot
    J$ in $I/I^2$.
    The image of $(I^2:\partial)\cap I$ is the annihilator of $\partial$ in
    $I/I^2$. This annihilator clearly contains $(I:\partial)\cdot I/I^2 =
    J\cdot I/I^2$. This shows that if the $S/I$-module $I/I^2$ is ``nice'', for
    example free, we should have an equality $(I^2:\partial)\cap I = I\cdot J$.
    More generally this equality is connected to the syzygies of
    $I/I^2$.
\end{remark}

In the remainder of this subsection we will prove that in several situations
the conditions of Theorem~\ref{ref:nonobstructedconds:thm} are satisfied.

\begin{cor}\label{ref:CIarenonobstructed:cor}
    We keep the notation and assumptions of
    Theorem~\ref{ref:nonobstructedconds:thm}. Suppose further
    that the algebra $\DS/I = \Apolar{f}$ is a complete intersection. Then the equivalent
    conditions of Theorem~\ref{ref:nonobstructedconds:thm} are satisfied.
\end{cor}

\begin{proof}
    Since $\DS/I$ is a complete intersection, the $\DS/I$-module $I/I^2$ is
    free, see e.g. \cite[Thm~16.2]{Matsumura_CommRing} and discussion above it or
    \cite[Ex~17.12a]{EisView}. It implies that $(I^2 : \partial) \cap I = (I : \partial)I
    = JI$, because $J = \Dan{\partial\hook f} = \{ s\in \DS\ |\ s
    \partial\hook f = 0\} = (\Dan{f} : \partial) = (I : \partial)$. Thus the
    condition from Point 4 of Theorem
    \ref{ref:nonobstructedconds:thm} is satisfied.
\end{proof}

\begin{example}\label{ref:14531case:example}
    If $A = \DS/I$ is a complete intersection, then it is
    smoothable and unobstructed
    (see Subsection~\ref{sss:smoothability}). The apolar algebras
    of monomials are complete intersections, therefore the assumptions of
    Theorem~\ref{ref:nonobstructedconds:thm} are satisfied e.g.~for $f
    =x_1^2x_2^2x_3$ and $\partial = \Dx_2^2$. Now
    Corollary~\ref{ref:CIarenonobstructed:cor} implies that the equivalent
    conditions of the Theorem are also satisfied, thus $x_1^2x_2^2x_3 +
    x_4^{d}x_1^2x_3 = (x_2^2x_3)(x_1^2 + x_4^d)$ is unobstructed for every $d\geq
    2$ (provided $\kchar = 0$ or $d \leq \kchar$).
    Similarly, $x_1^2x_2x_3 + x_4^2x_1$ is unobstructed
    and has Hilbert function $(1, 4, 5, 3, 1)$.
\end{example}

\begin{example}\label{ref:1441:example}
    Let $f = (x_1^2 + x_2^2)x_3$, then $\Dan{f} = (\Dx_1^2 - \Dx_2^2,
    \Dx_1\Dx_2, \Dx_3^2)$ is a complete intersection. Take $\partial =
    \Dx_1\Dx_3$, then $\partial\hook f = x_1$ and
    $\partial^2\hook f = 0$, thus $f + x_4^2\partial\hook f = x_1^2x_3 +
    x_2^2x_3 + x_4^2x_1$ is unobstructed. Note that by
    Remark~\ref{ref:Hilbfunccouting:rmk} the apolar algebra of this
    polynomial has Hilbert function $(1, 4, 4, 1)$.
\end{example}

\begin{prop}\label{ref:unobstructeddoubleray:prop}
    Let $f\in \DP$ be such that $\Apolar{f}$ is a complete
    intersection.

    Let $d$ be a natural number. Suppose that $\kchar = 0$ or $d\leq \kchar$.
    Take $\partial\in \DPut{Sf}{\DS}$ such that $\partial^2\hook f = 0$ and
    $\Apolar{\partial\hook f}$ is also a complete intersection.
    Let $g\in \DP[y]$ be the $d$-th ray sum $f$ with respect to $\partial$,
    i.e.~$g = f + y^{d} \partial\hook f$.

    Suppose that $\deg \partial\hook f > 0$.
    Let $\beta$ be the variable dual to $y$ and $\sigma\in \DSf$ be such that
    $\sigma\hook (\partial\hook f) = 1$. Take $\varphi := \sigma\beta\in
    \DPut{Sg}{\DT} = \DS[[\beta]]$.
    Let $h$ be any ray sum of $g$ with respect to $\varphi$, explicitly
    \[
        h = f + y^{d} \partial\hook f + z^my^{d-1}
    \]
    for some $m\geq 2$.

    Then the algebra $\Apolar{h}$ is
    unobstructed.
\end{prop}

\begin{proof}
    First note that $\varphi\hook g = y^{d-1}$ and so $\varphi^2\hook g =
    \sigma\hook y^{d-2} = 0$, since $\sigma\in
    \mathfrak{m}_{\DSf}$. Therefore indeed $h$ has the presented form.
    \def\DmmV{\mathfrak{m}_{\DSf}}%

    From Corollary \ref{ref:CIarenonobstructed:cor} it follows that
    $\Apolar{g}$ is unobstructed. Since $\varphi\hook g = y^{d-1}$,
    the algebra $\Apolar{\varphi\hook g}$ is unobstructed as well. Now by
    Theorem \ref{ref:nonobstructedconds:thm} it remains to prove that
    \begin{equation}\label{eq:maincontainment}
        (I_g^2:\varphi) \cap I_g \cap J_g^2 \subseteq I_g J_g,
    \end{equation}
    where
    $\DPut{Ig}{I_g} =
    \annn{\DSg}{g}, \DPut{Jg}{J_g} = \annn{\DSg}{\varphi\hook g}$.
    The rest of the proof is a technical verification of this claim.
    Denote $\DPut{If}{I_f} := \annn{\DSf}{f}$ and $\DPut{Jf}{J_{f}} := \annn{\DSf}{\partial\hook f}$;
    note that we take annihilators in $\DSf$.
    By Proposition \ref{ref:raysumideal:prop} we have $\DIg = \DIf\DT +
    \beta\DJf\DT + \DPut{spec}{(\beta^{d} - \partial)}\DT$.
    Consider $\gamma\in \DSg$ lying in $(\DIg^2 : \varphi) \cap \DIg \cap
    \DJg^2$. Write $\gamma = \gamma_0 + \gamma_1 \beta + \gamma_2 \beta^2 +
    \dots$ where $\gamma_i\in \DSf$, so they do not contain $\beta$. We will
    prove that $\gamma\in \DIg\DJg$.

    First, since $\Dspec^2 \in \DIg\DJg$ we may reduce powers of $\beta$ in $\gamma$ using this
    element and so we assume $\gamma_{i} = 0$ for $i\geq 2d$.
    Let us take $i < 2d$. Since $\gamma\in \DJg^2 =
    \pp{\annn{\DSg}{y^{d-1}}}^2 = \pp{\DmmV, \beta^d}^2$ we see that $\gamma_i\in
    \DmmV \subseteq \DJg$. For $i > d$ we have $\beta^i \in \DIg$, so
    that $\gamma_i \beta^i \in \DJg\DIg$ and we may
    assume $\gamma_i = 0$.
    Moreover, $\beta^d \gamma_d - \partial \gamma_d \in \DIg\DJg$ so we may also
    assume $\gamma_d = 0$, obtaining
    \[\gamma = \gamma_0 + \dots + \gamma_{d-1} \beta^{d-1}.\]
    From the explicit description of $\DIg$ in
    Proposition~\ref{ref:raysumideal:prop} it follows that $\gamma_i\in \DJf$
    for all $i$.

    Let $M = \DIg^{2} \cap \varphi\DT = \DIg^2 \cap \DJf\beta\DT$. Then for
    $\gamma$ as above we have $\gamma \varphi\in M$, so we will analyse the
    module $M$.
    Recall that
    \begin{equation}\label{eq:scarydecomposition}
        \DIg^2 = \DIf^2\cdot \DT + \beta \DIf \DJf\cdot \DT + \beta^2 \DJf^2\cdot \DT +
        \Dspec\DIf\cdot \DT + \Dspec\beta\DJf \cdot \DT + \Dspec^2\cdot \DT.
    \end{equation}
    We claim that
    \begin{equation}\label{eq:contains}
        M \subseteq \DIf^2\cdot \DT + \beta\DIf \DJf\cdot \DT + \beta^2
        \DJf^2\cdot \DT + \Dspec\beta\DJf\cdot \DT.
    \end{equation}
    We have $\DIg^2 \subseteq
    \DJf \cdot \DT + \Dspec^2\cdot\DT$, so
    if an element of $\DIg^2$ lies in
    $\DJf\cdot\DT$, then its coefficient standing next to $\Dspec^2$ in Presentation
    \eqref{eq:scarydecomposition} is an element of $\DJf$ by
    Lemma~\ref{ref:decompositionhomog:lem}.
    Since $\DJf \cdot
    \Dspec \subseteq \DIf + \beta\DJf$, we may ignore the term $\Dspec^2$:
    \begin{equation}\label{eq:lessscdec}
        M \subseteq \DIf^2\cdot \DT + \beta \DIf \DJf\cdot \DT + \beta^2 \DJf^2\cdot \DT +
        \Dspec\DIf\cdot \DT +  \Dspec\beta\DJf\cdot \DT.
    \end{equation}
    Choose an element of $M$ and let $i\in \DIf\cdot\DT$ be the coefficient of this
    element standing next to $\Dspec$. Since $\DIf\DT \cap \beta \DT \subseteq
    \DJf\DT$ we may assume that $i$ does not contain $\beta$, i.e. $i\in
    \DIf$.
    Now, if an element of the right hand side of \eqref{eq:lessscdec} lies in
    $\beta\cdot\DT$, then the coefficient $i$ satisfies
    $i\cdot \partial\in \DIf^2$, so that $i\in (\DIf^2 : \partial)$. Since
    $\DIf$ is a complete intersection ideal the $\DS/\DIf$-module
    $\DIf/\DIf^2$ is free, see Corollary~\ref{ref:CIarenonobstructed:cor} for
    references. Then we have $(\DIf^2: \partial) =
    (\DIf:\partial)\DIf$ and $i\in (\DIf:\partial)\DIf = \DIf\DJf$. Then
    $i\cdot \Dspec \subseteq \DIf^2 + \beta\cdot \DIf\cdot \DJf$ and so the
    Inclusion \eqref{eq:contains} is proved. We come back to the proof of
    proposition.

    From Lemma~\ref{ref:decompositionhomog:lem} applied to the ideal
    $\DJf^2\DT$ and the element $\beta\Dspec$ and the fact that $\beta\partial\DJf^2
    \subseteq I_g^2$ we compute
    that $M\cap \{ \delta\ |\ \deg_{\beta} \delta \leq d\}$ is
    a subset of $\DIf^2\cdot\DT + \beta\cdot \DIf \DJf\cdot\DT + \beta^2
    \DJf^2\cdot \DT$. Then $\gamma \varphi = \gamma \beta\sigma$ lies in this set, so that
    $\gamma_0 \in (\DIf\DJf : \sigma)$ and $\gamma_{n} \in (\DJf^2 : \sigma)$
    for $n > 1$. Since $\Apolar{f}$ and $\Apolar{\partial\hook f}$ are
    complete intersections, we have
    $\gamma_0 \in \DIf\DmmV$ and $\gamma_i \in
    \DJf\DmmV$ for $i \geq 1$.
    It follows that $\gamma\in \DIg\DmmV \subseteq \DIg\DJg$.
\end{proof}

\begin{example}\label{ref:1551:example}
    Let $f\in P$ be a polynomial such that $A = \Apolar{f}$ is a complete
    intersection. Take
    $\partial$ such that $\partial\hook f = x_1$ and $\partial^2\hook f = 0$.
    Then the apolar algebra of $f + y_1^{d}x_1 + y_{2}^{m}y_1^{d-1}$ is unobstructed
    for any $d, m\geq 2$ (less or equal to $\kchar$ if it is non-zero). In particular $g = f + y_1^2x_1 + y_2^2y_1$ is
    unobstructed.

    Continuing Example \ref{ref:1441:example}, if $f =
    x_1^2x_3 + x_2^2x_3$, then $x_1^2x_3 + x_2^2x_3 + x_4^2x_1 + x_5^2x_4$ is
    unobstructed. The apolar algebra of this polynomial has Hilbert function
    $(1, 5, 5, 1)$.

    Let $g = x_1^2x_3 + x_2^2x_3 + x_4^2x_1$, then $x_1^2x_3 + x_2^2x_3 +
    x_4^2x_1 + x_5^2x_4$ is a ray sum of $g$ with respect to $\partial =
    \Dx_4\Dx_1$. Let $I := \Dan{g}$ and $J := (I : \partial)$.
    In contrast with Corollary~\ref{ref:CIarenonobstructed:cor} and Example~\ref{ref:1441:example} one may
    check that all three terms $I$, $J^2$ and $(I^2 : \partial)$ are necessary to
    obtain equality in the inclusion \eqref{eq:maincontainment} for $g$ and $\partial$, i.e.~no two
    ideals of $I$, $J^2$, $(I^2 : \partial)$ have intersection equal to $IJ$.
\end{example}

\begin{example}\label{ref:144311:example}
    Let $f = x_1^5 + x_2^4$. Then the annihilator of $f$ in $k[\Dx_1, \Dx_2]$
    is a complete intersection, and this is true for every $f\in k[x_1, x_2]$. Let
    $g = f + x_3^2x_1^2$ be the second ray sum of $f$ with respect to
    $\Dx_1^3$ and $h = g + x_4^2x_3$ be the second ray sum of $g$ with
    respect to $\Dx_3\Dx_1^2$.
    Then the apolar algebra of
    \[h = x_1^5 + x_2^4 + x_3^2x_1^2 + x_4^2x_3\] is smoothable and not
    obstructed. It has Hilbert function $(1, 4, 4, 3, 1, 1)$.
\end{example}

\begin{remark}
    The assumption $\deg \partial\hook f > 0$ in
    Proposition~\ref{ref:unobstructeddoubleray:prop} is necessary:
    the polynomial $h = x_1x_2x_3 + x_4^2 + x_5^2x_4$ is obstructed, with
    length $12$ and tangent space dimension $67 > 12\cdot 5$ over
    $k = \mathbb{C}$. The polynomial $g$ is the fourth ray sum of $x_1 x_2 x_3$
    with respect to $\Dx_1\Dx_2\Dx_3$ and $h$ is the second
    ray sum of $g = x_1 x_2 x_3 + x_4^2$ with respect to $\Dx_4$, thus this
    example satisfies the assumptions of
    Proposition~\ref{ref:unobstructeddoubleray:prop} except for $\deg
    \partial\hook f > 0$. Note that in this case $\Dx_4^2\hook g \neq 0$.
\end{remark}

\section{Proof of Main Theorem and comments on the degree 14
case}\label{sec:proof}

\subsection{Preliminary results}\label{sss:parameterising}

Let $r\geq 1$ be a natural number and $V $ be a constructible subset of
$P_{\leq s}$. Assume that the apolar
algebra $\Apolar{f}$ has length $r$ for every closed
point $f\in V$. Then we may construct the incidence scheme $\{(f,
    \Apolar{f})\}\to V$ which is a finite flat family over $V$ and thus we obtain a morphism from $V$ to the (punctual) Hilbert
    scheme of $r$ points on an appropriate $\mathbb{P}^n$. See
    \cite[Prop~4.39]{JelMSc} for details.

    Consider $f\in \DP_{\leq s}$.
    The apolar algebra of $f$ has length at most $r$ if and only if the matrix
    of partials $\DS_{\leq s} f$ has rank at most $r$. This is a
    closed condition, so we obtain the following
    Remark~\ref{ref:semicontinuity:rmk}.
\begin{remark}\label{ref:semicontinuity:rmk}
    Let $s$ be a positive integer and $V \subseteq \DP_{\leq s}$ be a constructible subset. Then
    the set $U$, consisting of $f\in V$ such that the apolar algebra of $f$
    has the maximal length (among the elements of $V$), is open in $V$. In
    particular, if $V$ is irreducible then $U$ is also irreducible.
\end{remark}

\begin{example}\label{ref:semicondegthree:example}
    Let $\DP_{\geq 4} = k[x_1, \ldots
    ,x_n]_{\geq 4}$ be the space of polynomials that are sums of
    monomials of degree at least $4$.
    Suppose that the set $V \subseteq \DP_{\geq 4}$ parameterising algebras
    with fixed Hilbert function $H$ is irreducible. Then also the set $W$ of
    polynomials $f\in \DP$ such that $f_{\geq 4}\in V$ is irreducible. Let
    $e:= H(1)$ and suppose that the symmetric decomposition of $H$ has zero
    rows $\Dhdvect{s-3} = (0, 0, 0, 0)$ and $\Dhdvect{s-2} = (0, 0, 0)$, where
    $s = \deg f$.
    We claim that general element of $W$ corresponds to an algebra $B$ with Hilbert
    function: $H_{max} = H + (0, n-e, n-e, 0)$.
    Indeed, since we may only vary the degree three part of the polynomial,
    the function $H_B$ has the form $H + (0, a, a, 0) + (0, b, 0)$ for some
    $a, b$ such that $a + b \leq n - e$. Therefore algebras with Hilbert
    function $H_{max}$ are precisely the algebras of maximal possible length.
    Since $H_{max}$ is attained for $f_{\geq 4} + x_{e+1}^3 +
    \ldots  + x_n^3$, the claim follows from
    Remark~\ref{ref:semicontinuity:rmk}.
\end{example}

\subsection{Lemmas on Hilbert functions}

In the following $H_A$ denotes the Hilbert function of an algebra $A$.

\begin{lem}\label{ref:hilbertfunc:lem}
    Suppose that $A$ is a local Artin Gorenstein algebra of socle degree $s\geq 3$ such that
    $\Dhdvect{A, s-2} = (0, 0, 0)$. Then $\len A \geq 2\left(H_A(1) + 1\right)$.
    Furthermore, equality occurs if and only if $s = 3$.
\end{lem}

\begin{proof}
    Consider the symmetric decomposition $\Dhdvect{\bullet} = \Dhdvect{A,
    \bullet}$ of $H_A$.
    From symmetry we have $\sum_j \Dhd{0}{j} \geq 2 + 2\Dhd{0}{1}$ with
    equality only if $\Dhdvect{0}$ has no terms between $1$ and $s-1$~i.e.~when $s = 3$.
    Similarly $\sum_j \Dhd{i}{j}\geq 2\Dhd{i}{1}$ for all $1 \leq i < s-2$.
    Summing these inequalities we obtain
    \[
        \len A = \sum_{i<s-2} \sum_j \Dhd{i}{j} \geq 2 + \sum_{i<s-2} 2\Dhd{i}{1} = 2
        + 2H_A(1).\qedhere
    \]
\end{proof}

\begin{lem}\label{ref:trikofHilbFunc:lem}
    Let $A$ be a local Artin Gorenstein algebra of length at most $14$. Suppose
    that $4\leq H_A(1) \leq 5$. Then $H_A(2)\leq 5$.
\end{lem}

\begin{proof}
    Let $s$ be the socle degree of $A$.
    Suppose $H_A(2) \geq 6$. Then $H_{A}(3) + H_{A}(4) + \dots \leq 3$, thus
    $s\in \{3, 4, 5\}$. The cases $s = 3$ and $s = 5$ immediately lead to
    contradiction -- it is impossible to get the required symmetric
    decomposition. We will consider the case $s = 4$. In this case $H_A = (1,
    *, *, *, 1)$ and its symmetric decomposition is $(1, e, q, e, 1) + (0, m,
    m, 0) + (0, t, 0)$.
    Then $e = H_A(3) \leq 14 - 2 - 4 - 6 = 2$.
    Since $H_A(1) < H_A(2)$ we have $e < q$. This can
    only happen if $e = 2$ and $q = 3$. But then $14\geq \len A = 9 + 2m + t$,
    thus $m\leq 2$ and $H_A(2) = m + q \leq 5$. A contradiction.
\end{proof}

\begin{lem}\label{ref:14341notexists:lem}
    There does not exist a local Artin Gorenstein algebra with Hilbert
    function \[(1, 4, 3, 4, 1, \ldots , 1).\]
\end{lem}

\begin{proof}
    See \cite[pp.~99-100]{ia94} for the proof or \cite[Lem~5.3]{CJNpoincare} for a generalisation. We provide a sketch for
    completeness.
    Suppose such an algebra $A$ exists and fix its dual socle
    generator $f\in k[x_1, \ldots, x_4]_\Ddegf$ in the standard form. Let $I =
    \Dan{f}$.
    The proof relies on two observations. First, the leading term of $f$ is, up
    to a constant, equal to $x_1^\Ddegf$ and in fact we may take $f = x_1^\Ddegf +
    f_{\leq 4}$. Moreover from the symmetric decomposition it follows that the
    Hilbert functions of  $\Apolar{x_1^s + f_4}$ and $\Apolar{f}$ are equal. Second,
    $h(3) = 4 = 3^{\langle 2\rangle} = h(2)^{\langle 2\rangle}$ is the maximal growth, so arguing similarly as in
    Lemma~\ref{ref:P1gotzmann:lem} we may assume that the degree
    two part $I_2$ of the ideal of $\gr A$ is equal to $((\Dx_3,
    \Dx_4)\DS)_2$. Then any derivative of $\Dx_3\hook f_4$ is a derivative of
    $x_1^s$, i.e. a power of $x_1$. It follows that $\Dx_3\hook f_4$ itself is a
    power of $x_1$; similarly $\Dx_4\hook f_4$ is a power of $x_1$.
    It follows that $f_4\in x_1^3\cdot k[x_1,x_2,x_3,x_4] + k[x_1,
    x_2]$, but then $f_4$ is annihilated by a
    linear form, which contradicts the fact that $f$ is in
    the standard form.
\end{proof}

The following lemmas essentially deal with the limit-reducibility in the case
$(1, 4, 4, 3, 1, 1)$. Here the method is straightforward, but the cost
is that the proof is broken into several cases and quite long.

\begin{lem}\label{ref:144311Hilbfunc:lem}
    Let $f = x_1^5 + f_4$ be a polynomial such that $H_{\Apolar{f}}(2) <
    H_{\Apolar{f_4}}(2)$.
    Let $\DPut{tmpQ}{\mathcal{Q}} = \DS_2 \cap
    \Dan{x_1^5} \subseteq
    \DS_2$. Then $x_1^2\in \DtmpQ f_4$ and $\Dan{f_4}_{2} \subseteq \DtmpQ$.
\end{lem}

\begin{proof}
    Note that $\dim \DtmpQ f_4 \geq \dim \DS_2 f_4 - 1 = H_{\Apolar{f_4}}(2) -
    1$. If $\Dan{f_4}_{2} \not\subseteq \DtmpQ$, then there is a $q\in \DtmpQ$
    such that $\Dx_1^2 - q\in \Dan{f_4}$. Then $\DtmpQ f_4 = \DS_2 f_4$ and
    we obtain a contradiction.
    Suppose that $x_1^2\not\in \DtmpQ f_4$. Then the degree
    two partials of $f$ contain a direct sum of $k x_1^2$ and $\DtmpQ f_4$,
    thus they are at least $H_{\Apolar{f_4}}(2)$-dimensional, so that
    $H_{\Apolar{f}}(2)\geq H_{\Apolar{f_4}}(2)$, a
    contradiction.
\end{proof}

\begin{lem}\label{ref:144311caseCI:lem}
    Let $f = x_1^5 + f_4\in \DP$ be a polynomial such that $H_{\Apolar{f}} = (1, 3, 3, 3, 1, 1)$
    and $H_{\Apolar{f_4}} = (1, 3, 4, 3, 1)$. Suppose that $\Dx_1^3\hook f_4 =
    0$ and that $\pp{\Dan{f_4}}_2$ defines a complete intersection. Then
    $\Apolar{f_4}$ and $\Apolar{f}$ are complete intersections.
\end{lem}

\begin{proof}
    Let $I := \Dan{f_4}$.
    First we will prove that $\Dan{f_4} = (q_1, q_2, c)$, where $\langle q_1,
    q_2\rangle = I_2$ and $c\in I_3$. Then of course $\Apolar{f_4}$ is a complete intersection.
    By assumption, $q_1, q_2$ form a regular sequence. Thus there are no syzygies of
    degree at most three in the minimal resolution of $\Apolar{f_4}$. By
    the symmetry of the minimal resolution, see \cite[Cor 21.16]{EisView},
    there are no generators of degree at least four in the minimal generating
    set of $I$. Thus $I$ is generated in degree two and three. But
    $H_{\DS/(q_1, q_2)}(3) = 4 = H_{\DS/I}(3) + 1$, thus there is a
    cubic $c$, such that $I_3 = kc\oplus (q_1, q_2)_3$, then $(q_1, q_2, c) = I$, thus $\Apolar{f_4} = \DS/I$
    is a complete intersection.

    Let $\DPut{anntwo}{\mathcal{Q}}:= \Dan{x_1^5} \cap S_2 \subseteq S_2$.
    By Lemma~\ref{ref:144311Hilbfunc:lem} we have $q_1, q_2\in \Danntwo$, so
    that $\Dx_1^3\in I \setminus (q_1, q_2)$, then $I = (q_1, q_2, \Dx_1^3)$.
    Moreover, by the same Lemma, there exists $\sigma\in \Danntwo$ such that $\sigma\hook f_4 =
     x_1^2$.

    Now we prove that $\Apolar{f}$ is a complete intersection.
    Let $J := (q_1, q_2, \Dx_1^3 - \sigma) \subseteq \Dan{f}$.
    We will prove that $\DS/J$ is a complete intersection.
    Since $q_1$, $q_2$, $\Dx_1^3$ is a
    regular sequence, the set $\DS/(q_1, q_2)$ is a cone over a scheme of
    dimension zero and $\Dx_1^3$ does not
    vanish identically on any of its components. Since $\sigma$ has degree two, $\Dx_1^3 - \sigma$
    also does not vanish identically on any of the components of $\Spec
    \DS/(q_1, q_2)$, thus $\Spec \DS/J$ has dimension zero,
    so it is a complete intersection (see also \cite[Cor~2.4,
    Rmk~2.5]{Valabrega_FormRings}).
    Then the quotient by $J$ has length at most
    $\deg(q_1)\deg(q_2)\deg(\Dx_1^3 - \sigma) = 12 = \dimk \DS/\Dan{f}$. Since
    $J \subseteq \Dan{f}$, we have $\Dan{f} = J$ and
    $\Apolar{f}$ is a complete intersection.
\end{proof}

\begin{lem}\label{ref:144311casenotCI:lem}
    Let $f = x_1^5 + f_4 + g$, where $\deg g\leq 3$, be a polynomial such that
    $H_{\Apolar{f_{\geq 4}}} = (1, 3, 3, 3, 1, 1)$
    and $H_{\Apolar{f_4}} = (1, 3, 4, 3, 1)$. Suppose that $\Dx_1^3\hook f_4 =
    0$ and that $\pp{\Dan{f_4}}_2$ does not define a complete intersection.
    Then $\Apolar{f}$ is limit-reducible.
\end{lem}

\begin{proof}
    \def\spann#1{\langle #1 \rangle}%
    Let $\langle q_1, q_2\rangle = \pp{\Dan{f_4}}_2$. Since $q_1, q_2$ do not
    form a regular sequence, we have, after a linear transformation $\varphi$, two
    possibilities: $q_1 = \Dx_1\Dx_2$ and $q_2 = \Dx_1\Dx_3$ or $q_1 =
    \Dx_1^2$ and $q_2 = \Dx_1\Dx_2$. Let $\beta$ be the image of $\Dx_1$ under
    $\varphi$, so that $\beta^3\hook f_4 = 0$.

    Suppose first that $q_1 = \Dx_1\Dx_2$ and $q_2 = \Dx_1\Dx_3$. If $\beta$
    is up to constant equal to $\Dx_1$, then $\Dx_1\Dx_2, \Dx_1\Dx_3,
    \Dx_1^3\in \Dan{f_4}$, so that $\Dx_1^2$ is in the socle of
    $\Apolar{f_4}$, a contradiction. Thus we may assume, after another change
    of variables, that $\beta = \Dx_2$, $q_1 = \Dx_1\Dx_2$ and $q_2 = \Dx_1\Dx_3$.
    Then $f = x_2^5 + f_4 + \hat{g} = x_2^5 + x_1^4 + \hat{h} + \hat{g}$,
    where $\hat{h}\in k[x_1, x_3]$ and $\deg(\hat{g})\leq 3$. Then by
    Lemma~\ref{ref:topdegreetwist:lem} we may assume that $\Dx_1^2\hook f =
    0$, so $\Apolar{f}$ is limit-reducible by
    Proposition~\ref{ref:stretchedhavedegenerations:prop}. See also
    Example~\ref{ref:quarticlimitreducible:example} (the degree
    assumption in the Example can easily be modified).

    Suppose now that $q_1 = \Dx_1^2$ and $q_2 = \Dx_1\Dx_2$. If
    $\beta$ is not a linear combination of $\Dx_1, \Dx_2$, then we may assume $\beta
    = \Dx_3$. Let $m$ in $f_4$ be any monomial divisible by $x_1$. Since $q_1,
    q_2\in \Dan{f_4}$, we see that $m =\lambda x_1x_3^3$ for some $\lambda\in
    k$. But since $\beta^3\in
    \Dan{f_4}$, we have $m = 0$. Thus $f_4$ does not contain $x_1$, so
    $H_{\Apolar{f_4}}(1) < 3$, a contradiction. Thus $\beta\in \langle \Dx_1,
    \Dx_2\rangle$. Suppose $\beta = \lambda\Dx_1$ for
    some $\lambda\in k \setminus \{0\}$.
    Applying Lemma~\ref{ref:144311Hilbfunc:lem} to $f_{\geq 4}$ we see
    that $x_1^2$ is a derivative of $f_4$, so $\beta^2\hook f_4\neq 0$, but
    $\beta^2\hook f_4 = \lambda^2q_1\hook f_4 = 0$, a contradiction. Thus
    $\beta = \lambda_1 \Dx_1 + \lambda_2 \Dx_2$ and changing $\Dx_2$ we
    may assume that $\beta = \Dx_2$. This substitution does not change $\langle \Dx_1^2,
    \Dx_1\Dx_2\rangle$. Now we directly check that $f_4 =
    x_3^2(\kappa_1 x_1x_3 + \kappa_2 x_2^2 + \kappa_3 x_2x_3 + \kappa_4
    x_3^2)$, for some $\kappa_{\bullet}\in k$. Since $x_1$ is a derivative
    of $f$, we have $\kappa_1\neq 0$. Then a non-zero element
    $\kappa_2\Dx_1\Dx_3 - \kappa_1\Dx_2^2$ annihilates $f_4$. A contradiction
    with $H_{\Apolar{f_4}}(2) = 4$.
\end{proof}

\begin{lem}\label{ref:144311addingpartial:lem}
    Let a quartic $f_4$ be such that $H_{\Apolar{f_4}} = (1, 3, 3, 3, 1)$ and $\Dx_1^3\hook f_4 =
    0$. Then $H_{\Apolar{x_1^5 + f_4}}(2) \geq 4$.
\end{lem}

\begin{proof}
    \def\spann#1{\langle #1 \rangle}%
    Let $\DPut{anntwo}{\mathcal{Q}} = \Dan{x_1^5}_2 \subseteq \DS_2$.
    Let $I$ denote the apolar ideal of $f_4$.
    By Proposition~\ref{ref:thirdsecant:prop} we see that $I$ is minimally
    generated by three elements of degree two and two elements of degree four.
    In particular, there are no cubics in the generating set.
    Since $\Dx_1^3\in I_3$, there is an element in $\sigma\in I_2$ such that
    $\sigma = \Dx_1^2 - q$, where $q\in \Danntwo$. Therefore $\Danntwo\hook
    f_4 = \DS_2\hook f_4$.
    Moreover, $\sigma$ does not annihilate $x_1^2$, so that $x_1^2$ is not a
    partial of $f_4$.
    We see that $x_1^2$ and $\Danntwo\hook f_4$ are leading forms of partials of
    $x_1^5 + f_4$, thus
    \[H_{\Apolar{x_1^5 + f_4}}(2) \geq 1 + \dim(\Danntwo\hook f_4) = 1 +
        \dim(\DS_2\hook f_4) = 1 +
    H_{\Apolar{f_4}}(2) = 4.\qedhere\]
\end{proof}

\begin{remark}
    In the setting of Lemma~\ref{ref:144311addingpartial:lem}, it is not hard
    to deduce that $H_{\Apolar{x_1^5 + f_4}} = (1, 3, 4, 3, 1, 1)$ by
    analysing the possible symmetric decompositions. We do not need this
    stronger statement, so we omit the proof.
\end{remark}

\subsection{Proofs}

The following Proposition~\ref{ref:mainthmthree:prop} generalises results about algebras with Hilbert function
$(1, 5, 5, 1)$, obtained in \cite{JJ1551} and \cite{bertone2012division}.

\begin{prop}\label{ref:mainthmthree:prop}
    Let $A$ be a local Artin Gorenstein algebra of socle degree three and
    ${H_A(2)\leq 5}$. Then $A$ is smoothable.
\end{prop}

\begin{proof}
    Suppose that the Hilbert function of $A$ is $(1, n, e, 1)$.
    By Proposition \ref{ref:squares:prop} the dual socle generator of $A$ may
    be put in the form $f + x_{e+1}^2 + \dots + x_n^2$, where $f\in
    k[x_1,\dots,x_e]$. By repeated use of
    Example~\ref{ref:squareadding:example} we see that
    $A$ is a limit of algebras of the form $\Apolar{f} \times k^{\oplus n-e}$.
    Thus it is smoothable if and only if $B = \Apolar{f}$ is.

    Let $e:=H_A(2)$, then $H_B = (1, e, e, 1)$.
    If $H_B(1) = e \leq 3$ then $B$ is smoothable. It remains to consider $4\leq e\leq 5$.
    The set of points corresponding to algebras with Hilbert function $(1, e,
    e, 1)$ is irreducible in $\Hilb{e}{2e+2}$ by
    Remark~\ref{ref:semicontinuity:rmk} for obvious parameterisation (as mentioned in
    \cite[Thm~I, p.~350]{iaCompressed}), thus it will be enough to find a
    smooth point in this set which corresponds to a smoothable
    algebra.
    The cases $e = 4$ and
    $e = 5$ are considered in Example~\ref{ref:1441:example} and Example~\ref{ref:1551:example} respectively.
\end{proof}

\begin{remark}
    The claim of Proposition~\ref{ref:mainthmthree:prop} holds true if we
    replace the assumption ${H_A(2) \leq 5}$ by $H_A(2) = 7$, thanks to the
    smoothability of local Artin Gorenstein algebras with Hilbert function
    $(1, 7, 7, 1)$, see~\cite{bertone2012division}. We will not use this
    result.
\end{remark}

\begin{lem}\label{ref:14521case:lem}
    Let $A$ be a local Artin Gorenstein algebra with Hilbert function $H_A$
    beginning with $H_A(0) = 1$,
    $H_A(1) = 4$, $H_A(2) = 5$, $H_A(3) \leq 2$. Then $A$ is smoothable.
\end{lem}

\begin{proof}
    Let $f$ be a dual socle generator of $A$ in the standard form. From
    Macaulay's Growth Theorem it follows that $H_A(m) \leq 2$ for all $m\geq 3$,
    so that $H_A = (1, 4, 5, 2, 2,  \ldots , 2, 1,  \ldots , 1)$. Let $s$ be
    the socle degree of $A$.

    Let $\Dhdvect{A, s-2} = (0, q, 0)$ be the $(s-2)$-nd row of the symmetric
    decomposition of $H_A$. If $q > 0$, then by
    Example~\ref{ref:squareadding:example} we know that $A$ is limit-reducible; it
    is a limit of algebras of the form $B \times k$, such that $H_{B}(1) =
    H_A(1) - 1 = 3$. Then the algebra $B$ is smoothable
    (see~\cite[Prop~2.5]{cn09}),
    so $A$ is also smoothable. In the following we assume that $q = 0$.

    We claim that $\DPut{ffour}{f_{\geq 4}} \in k[x_1, x_2]$. Indeed, the symmetric
    decomposition of the Hilbert function is either $(1, 1,  \ldots , 1) + (0,
    1,  \ldots , 1, 0) + (0, 0, 1, 0, 0) + (0, 2, 2, 0)$ or $(1, 2,  \ldots ,
    2, 1) + (0, 0, 1, 0, 0) + (0, 2, 2, 0)$. In particular $e(s-3) =
    \sum_{i\geq 3} \Dhd{i}{1} = 2$, so that $\Dffour \in k[x_1,
    x_2]$ and $H_{\Apolar{\Dffour}}(1) = 2$, in particular $x_1$ is a derivative of $\Dffour $, i.e.~there exist
    a $\partial\in \DS$ such that $\partial\hook \Dffour  = x_1$. Then we may
    assume $\partial\in \DmmS^3$, so $\partial^2\hook f = 0$.

    Let us fix $\Dffour $ and consider the set of all polynomials of the
    form $h = \Dffour  + g$, where $g\in k[x_1, x_2, x_3, x_4]$ has degree at
    most three. By Example~\ref{ref:semicondegthree:example} the apolar
    algebra of a general such polynomial will have
    Hilbert function $H_A$. The set of polynomials $h$ with fixed
    $h_{\geq 4} = \Dffour $, such that $H_{\Apolar{h}} = H_A$, is irreducible.
    This set contains $h := \Dffour  + x_3^2x_1 + x_4^2x_3$. To finish the
    proof is it enough to show that $h$ is smoothable and unobstructed. Since
    $\Apolar{\Dffour }$ is a complete intersection, this
    follows from Example~\ref{ref:1551:example}.
\end{proof}

The following Theorem~\ref{ref:mainthmstretchedfive:thm} generalises numerous
earlier smoothability results on stretched (by Sally, see~\cite{SallyStretchedGorenstein}),
$2$-stretched (by Casnati and Notari, see \cite{CN2stretched}) and almost-stretched (by Elias and Valla, see
\cite{EliasVallaAlmostStretched}) algebras. It is important to understand
that, in contrast with the mentioned papers, we avoid a full classification of
algebras. In the course of the proof we give some partial classification.

\begin{thm}\label{ref:mainthmstretchedfive:thm}
    Let $A$ be a local Artin Gorenstein algebra with Hilbert function $H_A$ satisfying
    $H_A(2) \leq 5$ and $H_{A}(3)\leq 2$. Then $A$ is smoothable.
\end{thm}

\begin{proof}
    We proceed by induction on $\len A$, the case $\len A = 1$ being trivial.
    If $A$ has socle degree three, then the result follows from Proposition
    \ref{ref:mainthmthree:prop}. Suppose that $A$ has socle degree $s\geq 4$.

    Let $f$ be a dual socle generator of $A$ in the standard form.
    If the symmetric decomposition of
    $H_A$ has a term $\Dhdvect{s-2} = (0, q, 0)$ with $q\neq 0$, then by
    Example~\ref{ref:squareadding:example}, we have that $A$ is a limit of
    algebras of the form $B \times k$, where $B$ satisfies the assumptions
    $H_B(2) \leq 5$ and $H_B(2) \leq 2$ on the Hilbert function. Then $B$ is
    smoothable by induction, so also $A$ is smoothable. Further in the proof
    we assume that $\Dhdvect{A, s-2} = (0, 0, 0)$.

    We would like to understand the symmetric decomposition of the Hilbert
    function $H_A$ of $A$. Since $H_A$ satisfies the Macaulay growth
    condition (see Subsection \ref{ref:MacGrowth:sss}) it follows that $H_A =
    (1, n, m, 2, 2, \dots, 2, 1, \dots, 1)$, where the number of ``$2$'' is
    possibly zero. If follows that the possible symmetric decompositions
    of the Hilbert function are
    \begin{enumerate}
        \item $(1, 2, 2,  \ldots , 2, 1) + (0, 0, 1, 0, 0) + (0, n-3, n-3, 0)$,
        \item $(1, 1, 1 \ldots , 1, 1) + (0, 1, 1, \ldots , 1, 0) + (0, 0, 1,
            0, 0) + (0, n-3, n-3, 0)$,
        \item $(1, 1, 1 \ldots , 1, 1) + (0, 1, 2, 1, 0) + (0, n-3, n-3, 0)$,
        \item $(1,  \ldots , 1) + (0, n-1, n-1, 0)$,
        \item $(1, 2, \ldots ,2, 1) + (0, n-2, n-2, 0)$,
        \item $(1,  \ldots , 1) + (0,1,  \ldots , 1, 0) + (0, n-2, n-2, 0)$,
    \end{enumerate}
    and that the decomposition is uniquely determined by the Hilbert function.
    In all cases we have $H_A(1)\leq H_A(2)\leq 5$, so $f\in
    k[x_1, \ldots ,x_5]$.
    Let us analyse the first three cases. In each of them we have $H_A(2) = H_A(1) + 1$. If $H_A(1) \leq
    3$, then $A$ is smoothable, see \cite[Cor 2.4]{cn09}. Suppose $H_A(1) \geq
    4$. Since $H_A(2) \leq 5$,
    we have $H_A(2) = 5$ and $H_A(1) = 4$. In this case the result follows from
    Lemma~\ref{ref:14521case:lem} above.

    It remains to analyse the three remaining cases. The proof is similar to
    the proof of Lemma~\ref{ref:14521case:lem}, however here it
    essentially depends on induction.
    Let $\DPut{ffour}{f_{\geq 4}}$ be the sum of homogeneous components of $f$ which have
    degree at least four. Since $f$ is in the standard form, we have
    $\Dffour\in k[x_1, x_2]$. The decomposition of the Hilbert function $\Apolar{\Dffour}$ is
    one of the decompositions $(1, \ldots ,1)$, $(1,2 \ldots ,2,1)$, $(1, \ldots ,1) + (0, 1,  \ldots
    , 1, 0)$, depending on the decomposition of the Hilbert function of $\Apolar{f}$.

    Let us fix a vector $\hat{h} = (1, 2, 2, 2,  \ldots , 2, 1, 1, \ldots , 1)
    $ and take the sets
    \[
        V_1 := \left\{ f\in k[x_1, x_2]\ |\ H_{\Apolar{f}} = \hat{h}\right\}\mbox{
        and } V_2 := \left\{ f\in k[x_1, \ldots ,x_n]\ |\ f_{\geq 4}\in V_1 \right\}.
    \]
    By Proposition~\ref{ref:irreducibleintwovariables:prop} the set $V_1$ is
    irreducible and thus $V_2$ is also irreducible.
    The Hilbert function of the apolar algebra of a general member of $V_2$ is,
    by Example~\ref{ref:semicondegthree:example}, equal to $H_A$. It remains
    to show that the apolar algebra of this general member is
    smoothable.

    Proposition~\ref{ref:irreducibleintwovariables:prop} implies that the general
    member of $V_2$ has (after a nonlinear change of coordinates) the form
    $f + \partial\hook f$, where $f = x_1^{s} + x_2^{s_2} + g$ for some $g$ of
    degree at most three.  Using Lemma \ref{ref:topdegreetwist:lem} we may assume (after another
    nonlinear change of coordinates) that $\Dx_1^2\hook g = 0$.

    Let $B := \Apolar{x_1^{s} + x_2^{s_2} + g}$. We will show that $B$ is
    smoothable.
    Since $s \geq 4 = 2\cdot 2$
    Proposition~\ref{ref:stretchedhavedegenerations:prop} shows that $B$
    is limit-reducible. Analysing the fibers of
    the resulting degeneration, as
    in Example~\ref{ref:stretched:example}, we see that they have the form
    $B'\times k$, where $B' = \Apolar{\hat{f}}$ and $\hat{f} = \lambda^{-1}x_1^{s-1} + x_2^{s_2} + g$.
    Then $H_{B'}(3) = H_{\Apolar{\hat{f}_{\geq 4}}}(3) \leq 2$. Moreover,
    $\hat{f}\in k[x_1, \ldots ,x_5]$, so that $H_{B'}(1) \leq 5$. Now
    analysing the possible symmetric decompositions of $H_{B'}$, which are
    listed above, we see that $H_{B'}(2)\leq H_{B'}(1) = 5$.
    It follows from induction on the length that $B'$ is smoothable, thus
    $B'\times k$ and $B$ are smoothable.
\end{proof}

\begin{prop}\label{ref:mainthmsfour:prop}
    Let $A$ be a local Artin Gorenstein algebra of socle degree four satisfying $\len
    A\leq 14$. Then $A$ is smoothable.
\end{prop}

\begin{proof}
    We proceed by induction on the length of $A$. Then by
    Proposition~\ref{ref:mainthmthree:prop} (and the fact that all algebras of
    socle degree at most two are smoothable) we may assume that all algebras
    of socle degree \emph{at most} four and length less than $\len A$ are
    smoothable.

    If $\Dhdvect{A,
    1} = (0, q, 0)$ with $q\neq 0$, then by
    Example~\ref{ref:squareadding:example} the algebra $A$ is a limit of
    algebras of the form $A'
    \times k$, where $A'$ has socle degree four. Hence $A$ is
    smoothable. Therefore we assume $q = 0$. Then $H_A(1) \leq 5$ by Lemma
    \ref{ref:hilbertfunc:lem}. Moreover, we may assume $H_A(1) \geq 4$ since
    otherwise $A$ is smoothable by \cite[Cor 2.4]{cn09}.

    The symmetric
    decomposition of $H_A$ is $(1, n, m, n, 1) + (0, p, p, 0)$ for some $n, m,
    p$. By the fact that $n\leq 5$ and Stanley's result \cite[p.
    67]{StanleyCombinatoricsAndCommutative} we have $n\leq m$, thus $n\leq 4$
    and $H_A(2) \leq H_A(1)\leq 5$.
    Due to $\len A \leq 14$ we have four cases: $n = 1$, $2$, $3$,
    $4$ and five possible shapes of Hilbert functions: $H_A = (1, *, *, 1, 1)$,
    $H_A = (1, *, *, 2, 1)$, $H_A = (1, 4, 4, 3, 1)$, $H_A = (1, 4, 4, 4,
    1)$, $H_A = (1, 4, 5, 3, 1)$.

    The conclusion in the first two cases follows from Theorem
    \ref{ref:mainthmstretchedfive:thm}.
    In the remaining cases we first look for a suitable irreducible set of
    dual socle generators parameterising algebras with prescribed $H_A$.
    We examine the case $H_A = (1, 4, 4, 3,
    1)$. We claim that the set of $f\in \DP = k[x_1, x_2, x_3, x_4]$ in the
    standard form, which are generators of algebras with
    Hilbert function $H_A$ is irreducible. Since the leading form $f_4$ of
    such $f$ has Hilbert function $(1, 3, 3, 3, 1)$, the set of possible
    leading forms is irreducible by Proposition~\ref{ref:thirdsecant:prop}.
    Then the irreducibility follows from
    Example~\ref{ref:semicondegthree:example}. The irreducibility in the cases
    $H_A = (1, 4, 4, 4, 1)$ and $H_A = (1, 4, 5, 3, 1)$ follows similarly from
    Proposition~\ref{ref:fourthsecant:prop} together with
    Example~\ref{ref:semicondegthree:example}.
    In the first two cases we see that $f_4$ is a sum of powers of variables,
    then Example~\ref{ref:quarticlimitreducible:example} shows
    that the apolar algebra $A$ of a general $f$ is limit-reducible. More
    precisely,
    $A$ is limit of algebras of the form $A' \times k$, where $A'$ has socle
    degree at most four (compare Example~\ref{ref:stretched:example}). Then $A$ is smoothable.
    In the last case Example~\ref{ref:14531case:example} gives an unobstructed
    algebra in this irreducible set.
    This completes the proof.
\end{proof}

\begin{lem}\label{ref:144311final:lem}
    Let $A$ be a local Artin Gorenstein algebra with Hilbert function $(1, 4, 4, 3,
    1, 1)$. Then $A$ is limit-reducible.
\end{lem}

\begin{proof}
    Let $s = 5$ be the socle degree of $A$.  If $\Dhdvect{A, s-2}\neq (0,
    0, 0)$ then $A$ is limit-reducible by
    Example~\ref{ref:squareadding:example}, so we assume $\Dhdvect{A, s-2} =
(0, 0, 0)$.
    The only possible symmetric decomposition of the Hilbert function $H_A$ with $\Dhdvect{A,
    s-2} = (0, 0, 0)$ is
    \begin{equation}\label{eq:hfdecompositionprim}
        (1, 4, 4, 3, 1, 1) = (1, 1, 1, 1, 1, 1) + (0, 2, 2, 2, 0) + (0, 1, 1,
        0).
    \end{equation}
    Let us take a dual socle generator $f$ of $A$. We assume that $f$ is in
    the standard form: $f = x_1^5 + f_4 + g$, where $\deg g\leq 3$. Then
        $H_{\Apolar{x_1^5 + f_4}} = (1, 3, 3, 3, 1, 1)$. We analyse
    the possible Hilbert functions of $B = \Apolar{f_4}$. By Lemma~\ref{ref:topdegreetwist:lem} we
    may assume that $\Dx_1^3\hook f_4 = 0$. Suppose first that
    $H_{B}(1) \leq 2$. From \eqref{eq:hfdecompositionprim} it follows that
    $H_{\Apolar{x_1^5+f_4}}(1) = 3$, so that $H_B(1) = 2$ and we may assume that
    $f_4\in k[x_2, x_3]$. Then by
    Lemma~\ref{ref:topdegreetwist:lem} we may further assume $\Dx_1^2\hook (f
    - x_1^5) = 0$, then Proposition~\ref{ref:stretchedhavedegenerations:prop}
    asserts that $A = \Apolar{f}$ is limit-reducible.

    Suppose now that $H_B(1) = 3$. Since $x_1^5$ is annihilated by a codimension
    one space of quadrics, we have $H_B(2) \leq H_A(2) + 1$, so there are two
    possibilities: $H_B = (1, 3, 3, 3, 1)$ or $H_B = (1,
    3, 4, 3, 1)$. By Lemma~\ref{ref:144311addingpartial:lem} the case
$H_B = (1, 3, 3, 3, 1)$ is not possible, so that $H_B = (1, 3, 4, 3, 1)$. Now by Lemma~\ref{ref:144311casenotCI:lem} we may
    consider only the case when $\pp{\Dan{f_4}}_2$ is a complete
    intersection, then by Lemma~\ref{ref:144311caseCI:lem} we have that
    $\Apolar{x_1^5 + f_4}$ is a complete intersection. In this case we
    will actually prove that $A$ is smoothable.

    By Example~\ref{ref:semicondegthree:example} the set $W$ of
    polynomials $f$ with fixed leading polynomial $f_{\geq 4}$
    and Hilbert function $H_{\Apolar{f}} = (1, 4, 4, 3, 1, 1)$ is irreducible. Consider
        the apolar algebra $B$ of the polynomial $x_1^5 + f_4 +
        x_4^2x_1\in W$. By Proposition~\ref{ref:fibersofray:prop}, this algebra is
        the limit of smoothable algebras $\Apolar{x_1^5 + f_4} \times \Apolar{x_1}$,
        thus it is smoothable.
    By Corollary~\ref{ref:CIarenonobstructed:cor} the algebra $B$ is
        unobstructed. Thus apolar algebra of every element of $W$ is
    smoothable; in particular $A$ is smoothable.
\end{proof}

Now we are ready to prove Theorem~\ref{ref:mainthmsmoothable} which is the
algebraic counterpart of Theorems~\ref{ref:mainthm13degree}
and~\ref{ref:mainthm14degree}.

\begin{thm}\label{ref:mainthmsmoothable}
    Let $A$ be an Artin Gorenstein algebra of length at most $14$.
    Then either $A$ is smoothable or
    it is local with Hilbert function $(1, 6, 6, 1)$. In particular, if $A$ has length
    at most $13$, then $A$ is smoothable.
\end{thm}

\begin{proof}
    By the discussion in Section~\ref{sss:smoothability} it is enough to consider
local algebras.
    \def\Dh#1{H(#1)}%
    Let $A$ be a local algebra of length at most $14$ and of socle degree $s$.
    By $H$ we denote the Hilbert function of $A$.
    As mentioned in Subsection~\ref{sss:smoothability} it is enough to prove
    $A$ is limit-reducible.
    On the contrary, suppose that $A$ is strongly non-smoothable in the sense of Definition
    \ref{ref:limitreducible:def}. By Example~\ref{ref:squareadding:example} we
    have $\Dhdvect{A, s-2} = (0, 0, 0)$. Then by Lemma \ref{ref:hilbertfunc:lem} we see that either $H = (1,
    6, 6, 1)$ or $\Dh{1} \leq 5$. It is enough to consider $\Dh{1}\leq 5$.
    If $s = 3$ then $\Dh{2} \leq \Dh{1} \leq 5$, so by Proposition
    \ref{ref:mainthmthree:prop} we may assume $s > 3$. By Proposition
    \ref{ref:mainthmsfour:prop} it follows that
    we may consider only $s\geq 5$.

    If $\Dh{1}\leq 3$ then $A$ is smoothable by \cite[Cor 2.4]{cn09}, thus we
    may assume $\Dh{1} \geq 4$. By Lemma \ref{ref:trikofHilbFunc:lem} we
    see that $\Dh{2} \leq 5$. Then by Theorem
    \ref{ref:mainthmstretchedfive:thm} we may reduce to the
    case $\Dh{3} \geq 3$. By Macaulay's Growth Theorem we have $\Dh{2} \geq 3$.
    Then $\sum_{i>3} \Dh{i} \leq 14 - 11$, so we are left with several
    possibilities: $H = (1, 4, 3, 3, 1, 1, 1)$, $H = (1, 4, 3, 3, 2, 1)$ or
    $H = (1, *, *, *, 1, 1)$.
    In the first two cases it follows from the symmetric decomposition that
    $\Dhdvect{A, s-2} \neq (0, 0, 0)$ which is a contradiction. We examine the
    last case.
    By Lemma \ref{ref:14341notexists:lem} there does not exist an algebra with Hilbert function $(1, 4, 3,
    4, 1, 1)$. Thus the only possibilities are $(1, 4, 3, 3, 1, 1)$,
    $(1, 5, 3, 3, 1, 1)$ and $(1, 4, 4, 3, 1, 1)$. Once more, it can be
    checked directly
    that in the first two cases $\Dhdvect{A, s-2} \neq (0, 0, 0)$. The
    last case is the content of Lemma~\ref{ref:144311final:lem}.\qedhere
\end{proof}

\begin{remark}\label{ref:1661:rmk}
    Assume $\kchar = 0$.
    In \cite{emsalem_iarrobino_small_tangent_space} Emsalem and Iarrobino
    analysed the tangent space to the Hilbert scheme.
    Iarrobino and Kanev claim that using Macaulay they are able to check
        that the tangent space to $\Hilb{14}{6}$ has dimension
    $76$ at a point corresponding to a general local Gorenstein algebra $A$ with Hilbert
    function $(1, 6, 6, 1)$, see \cite[Lem~6.21]{iakanev}, see also \cite{cn10}
    for further details.  Since $76 < (1 + 6 + 6 + 1) \cdot 6$ this shows
    that $A$  is non-smoothable. Moreover, since all algebras of degree at most
    $13$ are smoothable, $A$  is strongly non-smoothable.
\end{remark}

To prove Theorem~\ref{ref:mainthm14degree}, we need to show that the
non-smoothable part of $\HilbGor_{14} \mathbb{P}^n$ (for $n\geq 6$) is
irreducible. The algebraic version of (a generalisation of) this statement is the following
lemma.
\begin{lem}\label{ref:1661irreducibility:lem}
    Let $n\geq m$ be natural numbers and $V \subseteq \DP_{\leq 3} = k[x_1,
    \ldots ,x_n]_{\leq 3}$ be the set of $f\in \DP$ such
    that $H_{\Apolar{f}} = (1, m, m, 1)$. Then $V$ is constructible
    and irreducible.
\end{lem}
\begin{proof}
    Let $V_{gr} = V \cap \DP_{3}$ denote
    the set of \emph{graded} algebras with Hilbert function $(1, m, m, 1)$.
    This is a constructible subset of $\DP_3$.
    To an element $f_3\in V_{gr}$ we may associate the tangent space to
    $\Apolar{f_3}$, which is isomorphic to $S_2\hook f_3$. We
    define
    \[\{(f_3, [W])\in V_{gr} \times
        \operatorname{Gr}(m, n) \ |\ W \supseteq S_2\hook f_3\},\]
        which is an open subset in a vector bundle $\{(f_3, [W])\in \DP_{3} \times
        \operatorname{Gr}(m, n) \ |\ W \supseteq S_2\hook f_3\}$ over
        $\operatorname{Gr}(m, n)$, given by the condition $\dim S_2\hook f_3 \geq m$.
    Let $f\in V$ and write it as $f = f_3 + f_{\leq 2}$,
    where $\deg f_{\leq 2} \leq 2$. Then $H_{\Apolar{f_3}} = (1, m, m, 1)$.
    Therefore we obtain a morphism $\varphi:V\to V_{gr}$ sending $f$ to $f_3$.
    We will analyse its fibers. Let $f_3\in V_{gr}$ and $f = f_3 + f_{\leq
    2}\in \DP_{\leq 3}$,
    where $\deg f_{\leq 2}\leq 2$. Then $H_{\Apolar{f}} = (1, M, m, 1)$ for
    some $M\geq m$. Moreover $M = m$ if and only if $\alpha\hook f_{\leq 2}$
    is a partial of $f_3$ for every $\alpha$ annihilating $f_3$.  The
    fiber of $\varphi$ over $f$ is an affine subspace of $P_{\leq 2}$ defined by these conditions
    and the morphism
    \[
        \{(f = f_3 + f_{\leq 2}, [W])\in V \times
        \operatorname{Gr}(m, n) \ |\ W \supseteq S_2\hook f_3\}\to \{(f_3, [W])\in V_{gr} \times
        \operatorname{Gr}(m, n) \ |\ W \supseteq S_2\hook f_3\}
    \]
    is a projection from a vector bundle, which is thus irreducible. Since $V$
    admits a surjection from this bundle, it is irreducible as
    well. Moreover, the above shows
    that $V$ is constructible.
\end{proof}

\begin{proof}[Proof of Theorems~\ref{ref:mainthm13degree}
    and~\ref{ref:mainthm14degree}]
    The locus of points of the Hilbert scheme corresponding to smooth
    (i.e.~reduced) algebras of length $d$ is irreducible,
    as an image of an open subset of the $d$--symmetric product of
    $\mathbb{P}^n$, and
    smooth.
    The locus of points corresponding to smoothable algebras is
    the closure of the aforementioned locus, so it is also irreducible. If
    $d\leq 13$ or $d\leq 14$ and $n\leq 5$, this locus is the whole Hilbert
    scheme by Theorem~\ref{ref:mainthmsmoothable} and the claim follows.

    Now consider
    the case $d = 14$ and $n\geq 6$. Let $\mathcal{V}$ be the set of points of
    the Hilbert scheme
    corresponding to local Gorenstein algebras with Hilbert function $(1, 6,
    6, 1)$. By Remark~\ref{ref:1661:rmk} these are the only non-smoothable algebras of length $14$, thus they
    deform only to local algebras with the same Hilbert function. Therefore,
    $\mathcal{V}$ is a sum of irreducible components of the Hilbert scheme. We
    will prove that $\mathcal{V}$ is an irreducible set, whose general point
    is smooth.

    Let $\mathcal{V}_p \subseteq \mathcal{V}$ denote the set
    consisting of schemes
    supported at a fixed point $p\in\mathbb{P}^n$. Then
    $\mathcal{V}$ is dominated by a set $\mathcal{V}_p \times
    \mathbb{P}^n$. Note that an irreducible scheme supported at a point $p$
    may be identified with a Gorenstein quotient of the power series ring having Hilbert
    function $(1, 6, 6, 1)$. These quotients are parameterised by the dual
    generators. More precisely, the set of $V$ of $f\in k[x_1, \ldots ,
    x_n]_{\leq 3}$ such that $H_{\Apolar{f}} = (1, 6, 6, 1)$ gives a morphism
    \[V\to \mathcal{V}_p \subseteq \HilbGor_{14} \mathbb{P}^n\]
    which sends $f$ to $\Spec \Apolar{f}$
    supported at $p$ (see subsection~\ref{sss:parameterising}). Since $V\to \mathcal{V}_p$ is surjective and $V$ is
    irreducible by Lemma~\ref{ref:1661irreducibility:lem}, we see that
    $\mathcal{V}_p$ is irreducible. Then $\mathcal{V}$ is irreducible as well.

    Take a smooth point of $\HilbGor_{14} \mathbb{P}^6$ which corresponds to
    an algebra $A$ with Hilbert function $(1, 6, 6, 1)$. Then any point of
    $\HilbGor_{14} \mathbb{P}^n$ corresponding to an embedding $\Spec A
    \subseteq \mathbb{P}^n$ is smooth by \cite[Lem 2.3]{cn09}. This concludes the
    proof.
\end{proof}

\section{Acknowledgements}

We wish to express our thanks to A.A.~Iarrobino and P.M.~Marques for
inspiring conversations. Moreover we are also sincerely grateful to
W.~Buczy\'nska and J.~Buczy\'nski for their care, support and hospitality during
the preparation of this paper. We also thank J.~Buczy\'nski for explaining the
    proof of Proposition~\ref{ref:irreducibility_of_dual_socle:prop}.
    We thank the referee for careful reading and suggesting a number of
    improvements.
    The examples were obtained with the help of
Magma computing software, see~\cite{Magma}.

\bigskip
\noindent
Gianfranco Casnati,\\
Dipartimento di Scienze Matematiche,  Politecnico di Torino, \\
corso Duca degli Abruzzi 24, 10129 Torino, Italy \\
e-mail: {\tt gianfranco.casnati@polito.it}

\bigskip
\noindent
Joachim Jelisiejew,\\
Faculty of Mathematics, Informatics, and Mechanics, University of Warsaw,\\
Banacha 2, 02-097 Warsaw, Poland\\
\texttt{j.jelisiejew@mimuw.edu.pl}

\bigskip
\noindent
Roberto Notari, \\
Dipartimento di Matematica, Politecnico di Milano,\\
via Bonardi 9, 20133 Milano, Italy\\
e-mail: {\tt roberto.notari@polimi.it}

\end{document}